\documentclass[10pt,twoside]{article}
\usepackage{mathrsfs}
\usepackage{amsmath}
\usepackage{amssymb}
\usepackage{fancyhdr}
\usepackage{latexsym}
\usepackage{bbding}
\usepackage{mathrsfs}
\usepackage{wasysym}
\usepackage{multicol,graphics}

\newtheorem{theorem}{Theorem}[section]
\newtheorem{lemma}[theorem]{Lemma}

\newtheorem{definition}[theorem]{Definition}

\newtheorem{proposition}[theorem]{Proposition}
\numberwithin{equation}{section}
\newenvironment{proof}[1][Proof]{\noindent\textbf{#1.} }{\hfill $\Box$}
\allowdisplaybreaks

\makeatletter
\begin{document}
\title{{Well-posedness and Gevrey Analyticity of the Generalized Keller-Segel System in Critical Besov Spaces}\footnote{This paper is partially supported by the National Natural Science Foundation of China (11371294), the Fundamental Research Funds for
the Central Universities (2014YB031) and the Fundamental Research Project of Natural Science in Shaanxi Province--Young Talent Project (2015JQ1004).}}
\author{Jihong Zhao\\
[0.2cm] {\small College of Science, Northwest A\&F
University, Yangling, Shaanxi 712100,  P. R. China}\\
[0.2cm] {\small E-mail: jihzhao@163.com, zhaojih@nwsuaf.edu.cn}}
\date{}
\maketitle

\begin{abstract}
In this paper, we study the Cauchy problem for the generalized Keller-Segel system with the cell
diffusion being ruled by fractional diffusion:
\begin{equation*}
\begin{cases}
  \partial_{t}u+\Lambda^{\alpha}u+\nabla\cdot(u\nabla \psi)=0\quad &\mbox{in}\ \
  \mathbb{R}^n\times(0,\infty),\\
   -\Delta \psi=u\quad &\mbox{in}\ \
  \mathbb{R}^n\times(0,\infty),\\
  u(x,0)=u_0(x), \ \  &\mbox{in}\ \ \mathbb{R}^n.
\end{cases}
\end{equation*}
In the case that $1<\alpha\leq 2$, we prove local well-posedness for any initial data and global well-posedness for small initial data in critical Besov spaces $\dot{B}^{-\alpha+\frac{n}{p}}_{p,q}(\mathbb{R}^{n})$ with $1\leq p<\infty$,  $1\leq q\leq \infty$,  and analyticity of  solutions for initial data  $u_{0}\in \dot{B}^{-\alpha+\frac{n}{p}}_{p,q}(\mathbb{R}^{n})$ with $1< p<\infty$,  $1\leq q\leq \infty$. Moreover, the global existence and analyticity of solutions with small initial data in critical Besov spaces $\dot{B}^{-\alpha}_{\infty,1}(\mathbb{R}^{n})$ is also established. In the limit case that $\alpha=1$, we prove global well-posedness for small initial data in critical Besov spaces $\dot{B}^{-1+\frac{n}{p}}_{p,1}(\mathbb{R}^{n})$ with $1\leq p<\infty$ and $\dot{B}^{-1}_{\infty,1}(\mathbb{R}^{n})$, and show analyticity of solutions for small initial data in $\dot{B}^{-1+\frac{n}{p}}_{p,1}(\mathbb{R}^{n})$ with $1<p<\infty$ and $\dot{B}^{-1}_{\infty,1}(\mathbb{R}^{n})$, respectively.

\textbf{Keywords}:  Generalized Keller-Segel system; chemotaxis model;  well-posedness; Gevrey analyticity;  decay

\textbf{2010 AMS Subject Classification}: 35B65, 35K15,  35M11, 92C17
\end{abstract}

\section{Introduction}
In this paper, we are concerned  with  the  nonlinear nonlocal evolution equations generalizing
the well-known Keller-Segel model of chemotaxis:
\begin{equation}\label{eq1.1}
\begin{cases}
  \partial_{t}u+\Lambda^{\alpha}u+\nabla\cdot(u\nabla \psi)=0\quad &\mbox{in}\ \
  \mathbb{R}^n\times(0,\infty),\\
   -\Delta \psi=u\quad &\mbox{in}\ \
  \mathbb{R}^n\times(0,\infty),\\
  u(x,0)=u_0(x), \ \  &\mbox{in}\ \ \mathbb{R}^n.
\end{cases}
\end{equation}
where $n\geq2$, $u$ and $\psi$ are two unknown functions which stand for the cell density and the concentration of the chemical attractant, respectively, and the anomalous (normal) diffusion is modeled by a fractional power
of the Laplacian with  $1\leq\alpha\leq2$. The positive operator $\Lambda^{\alpha}=(-\Delta)^{\frac{\alpha}{2}}$ is defined by
\begin{equation*}
  \Lambda^{\alpha}f(x):
    =2^{\alpha}\pi^{-\frac{n}{2}}\frac{\Gamma(\frac{n+\alpha}{2})}{\Gamma(-\frac{\alpha}{2})}\int_{\mathbb{R}^{n}}\frac{f(x-y)}{|y|^{n+\alpha}}dy.
\end{equation*}
A simple alternative representation
is given through the Fourier transform as
$\Lambda^{\alpha}f=\mathcal{F}^{-1}[|\xi|^{\alpha}\mathcal{F}f(\xi)]$,
where
$\mathcal{F}$ and $\mathcal{F}^{-1}$ are the Fourier transform and the inverse Fourier transform, respectively.

Obviously, the choice $\alpha=2$  in the system \eqref{eq1.1}  corresponds to a simplified system of
\begin{equation}\label{eq1.2}
\begin{cases}
  \partial_{t}u-\Delta u=-\nabla\cdot(u\nabla \psi)   &\mbox{in}\ \
  \mathbb{R}^n\times(0,\infty),\\
   \partial_{t}\psi-\Delta \psi=u-\psi   &\mbox{in}\ \
  \mathbb{R}^n\times(0,\infty),\\
  u(x,0)=u_0(x), \ \ \psi(x,0) =\psi_{0}(x) &\mbox{in}\ \ \mathbb{R}^n.
\end{cases}
\end{equation}
The system \eqref{eq1.2} is  a mathematical model of chemotaxis, which is formulated  by E.F. Keller and L.A. Segel \cite{KS70} in 1970, while it is also connected with astrophysical models of
gravitational self-interaction of massive particles in a cloud or a nebula,  see Biler, Hilhorst and Nadzieja \cite{BHN94}.

In biology, chemotaxis is the directed movement of an organism in response to ambient chemical gradients that are often segregated by the cells themselves. The
system \eqref{eq1.2} describes the manner in which cellular
slime molds aggregate owing to the motion of the cells, which move towards
higher concentration of a chemical substance which they produce themselves. In those cases where the chemical products are
attractive (and they are called chemoattractants), they lead to the phenomenon known
as chemotactic collapse: the cells accumulate in small regions of space giving rise to high
density configurations. This phenomenon exhibits that the system \eqref{eq1.2} admits  finite time blowup solutions for large enough initial data.   It was actually conjectured by Childress and Percus \cite{CP84} that in a two-dimensional domain $\Omega\subset\mathbb{R}^{2}$,
there exists a threshold $c_{0}$ such that if  the initial mass $m=\int_{\Omega}u_{0}(x)dx<c_{0}$, then the solution exists globally in time, while if $m=\int_{\Omega}u_{0}(x)dx>c_{0}$, then the solution blows up in finite time.  For various simplified versions
of the Keller-Segel system \eqref{eq1.2}, the conjecture has been essentially verified, see \cite{H03, H04} for a comprehensive review of these aspects.  Jager and Luckhaus \cite{JL92} considered the system \eqref{eq1.2} with Neumann boundary conditions  in a bounded domain $\Omega\subset\mathbb{R}^{2}$, and showed that  for sufficiently small $\frac{1}{|\Omega|}\int_{\Omega}u_{0}(x)dx$, there exists a unique smooth global positive solution, while for large $\frac{1}{|\Omega|}\int_{\Omega}u_{0}(x)dx$, there exists radial solutions which explode in finite time.  Herrero and Val\'{a}zquez \cite{HV961, HV962} studied the system \eqref{eq1.2} with  no-flux boundary conditions on a disk, and showed by the method of  matched asymptotic expansion that there exists a nonnegative radial initial data $(u_{0}, \psi_{0})$ with  $\int_{\Omega}u_{0}(x)dx>8\pi$ such that the solution $(u,\psi)$ corresponding to the initial data $(u_{0}, \psi_{0})$ blows up only at the origin in finite time and $u$ has a Dirac delta-type singularity at the origin.  Biler \cite{B98},  Gajewski and Zacharias\cite{GZ98}, Nagai, Senba and Yoshida \cite{NSY97}  subsequently proved  global existence of nonnegative solution under the condition $\int_{\Omega}u_{0}(x)dx<4\pi$, and existence of radial solutions on a disc under the condition $\int_{\Omega}u_{0}(x)dx<8\pi$.  Moreover, there exists  a detailed description of
the asymptotic behaviour of solutions of \eqref{eq1.2} in the case $\int_{\Omega}u_{0}(x)dx<8\pi$ to \cite{BDP06}, in the limit case $\int_{\Omega}u_{0}(x)dx=8\pi$ to \cite{BCM07} and in the radially symmetric case to \cite{BKLN961, BKLN962}. For more results related to this topic, we refer the reader to see  \cite{BKZ14, DNR98, LR091, LR092, N00, Y97}.

Since the chemical concentration $\psi$ is determined by the Poisson equation, the second equation of \eqref{eq1.1}, gives rise to the coefficient $\nabla\psi$
in the first equation of \eqref{eq1.1},  when $\psi$ is represented as the volume potential of $u$:
\begin{equation*}
\psi(x,t)=(-\Delta)^{-1}u(x,t)=
\begin{cases}
\frac{1}{n(n-2)\omega_{n}}\int_{\mathbb{R}^{n}}\frac{u(y,t)}{|x-y|^{n-2}}dy,
\quad n\geq 3,\\
-\frac{1}{2\pi}\int_{\mathbb{R}^{2}}u(y,t)\log|x-y|dy,
\quad n=2,
\end{cases}
\end{equation*}
where $\omega_{n}$ denotes the surface area of the unit sphere in
$\mathbb{R}^{n}$, the system \eqref{eq1.1} is essentially equivalent to
the following differential-integral Fokker-Planck system:
\begin{equation}\label{eq1.3}
       u=e^{-t\Lambda^{\alpha}}u_0-\int_0^te^{-(t-\tau)\Lambda^{\alpha}}\nabla\cdot[u\nabla(-\Delta)^{-1}u]d\tau.
\end{equation}
where $e^{-t\Lambda^{\alpha}}:=\mathcal{F}^{-1}[e^{-t|\xi|^{\alpha}}\mathcal{F}]$. We may find the solution of \eqref{eq1.3} by using the
contraction mapping argument for the mapping $u\mapsto \mathbb{F}(u)$ with
\begin{equation*}
      \mathbb{F}(u):=e^{-t\Lambda^{\alpha}}u_0-\int_0^te^{-(t-\tau)\Lambda^{\alpha}}\nabla\cdot[u\nabla(-\Delta)^{-1}u](\tau)d\tau.
\end{equation*}
The invariant space for solving the integral equation \eqref{eq1.3} requires us to analyze the scaling invariance property of the system \eqref{eq1.1}.  Set
\begin{equation*}
  u_{\lambda}(x,t):=\lambda^{\alpha}u(\lambda x, \lambda^{\alpha}t),\ \    \psi_{\lambda}(x,t):=\lambda^{\alpha-2}\psi(\lambda x, \lambda^{\alpha}t).
\end{equation*}
Then  if $u$ solves \eqref{eq1.1} with initial data $u_{0}$ ($\psi$ can be determined by $u$), so does
$u_{\lambda}$ with initial data $u_{0\lambda}$ ($\psi_{\lambda}$ can be determined by $u_{\lambda}$),  where $u_{0\lambda}(x):=\lambda^{\alpha}u_{0}(\lambda x)$. In particular, the norm of $u_{0}\in\dot{B}^{-\alpha+\frac{n}{p}}_{p,q}(\mathbb{R}^{n})$ ($1\leq p,q\leq\infty$) is scaling invariant under the above change of scale.

Note that in the case of classical Brownian diffusion $\alpha=2$, the solvability of the systems \eqref{eq1.1} has been relatively well-developed in various
classes of functions and distributions,  such as the Lebesgue space $L^{1}(\mathbb{R}^{n})\cap L^{\frac{n}{2}}(\mathbb{R}^{n})$  by Corrias, Perthame and Zaag \cite{CPZ04}, the Sobolev space $L^{1}(\mathbb{R}^{n})\cap W^{2,2}(\mathbb{R}^{n})$ by Kozono and Sugiyama \cite{KS08},  the Hardy space $\mathcal{H}^{1}(\mathbb{R}^{2})$ by Ogawa and Shimizu \cite{OS08}, the Besov space $\dot{B}^{0}_{1,2}(\mathbb{R}^{2})$ by Ogawa and Shimizu \cite{OS10},  the Besov space $\dot{B}^{-2+\frac{n}{p}}_{p,\infty}(\mathbb{R}^{n})$ and Fourier-Herz space $\dot{\mathcal{B}}^{-2}_{2}(\mathbb{R}^{n})$ by Iwabuchi \cite{I11}, and the pseudomeasure space $\mathcal{PM}^{n-2}(\mathbb{R}^{n})$ by Biler, Cannone, Guerra and Karch \cite{BCGK04}. We refer the reader to see  Lemari\'{e}-Rieusset \cite{L13} and the references therein for more results.

For general fractional diffusion case $1<\alpha<2$, the system \eqref{eq1.1} was first studied by Escudero in \cite{E06},  where it was used to describe the spatiotemporal distribution of a population density of random walkers undergoing L\'{e}vy flights. Moreover, the author proved that the one-dimensional system \eqref{eq1.1} possesses global in time solutions not
only in the case of $\alpha=2$ but also in the case $1<\alpha<2$.  Biler and Karch \cite{BK10}  proved existence and nonexistence of global in time solutions of \eqref{eq1.1}  in critical Lebesgue space $L^{\frac{n}{\alpha}}(\mathbb{R}^{n})$  for $1<\alpha<2$.  Biler and Wu \cite{BW09} established  global well-posedness of the system \eqref{eq1.1} with small initial data in the critical  Besov spaces $\dot{B}^{1-\alpha}_{2,q}(\mathbb{R}^{2})$ for $1<\alpha<2$. Wu and Zheng \cite{WZ11} proved a local well-posedness  with any initial data and global well-posedness with small initial data in critical Fourier-Herz
space $\mathcal{\dot{B}}^{2-2\alpha}_{q}(\mathbb{R}^{n})$ for $1<\alpha\leq2$ and $2\leq q\leq \infty$, and proved ill-posedness in  $\mathcal{\dot{B}}^{-2}_{q}(\mathbb{R}^{n})$ and  $\dot{B}^{-2}_{\infty, q}(\mathbb{R}^{n})$ with $\alpha=2$ and $2<q\leq\infty$.  Zhai \cite{Z10} proved the global existence, uniqueness and stability of solutions with small initial data in critical Besov spaces with general potential type nonlinear term.
Parts of these results were also proved for the system of two fractional power dissipative equations,
please refer to \cite{BW09, MYZ08, WY08} and the references therein.

In this paper, we aim at studying well-posedness and  Gevrey analyticity of the generalized Keller-Segel system \eqref{eq1.1} with initial data in critical Besov space $\dot{B}^{-\alpha+\frac{n}{p}}_{p,q}(\mathbb{R}^{n})$ for $1\leq \alpha\leq2$ and $1\leq p,q\leq\infty$.  The first novelty of this paper is that we resort the Fourier localization technique and the Bony's paraproduct theory to  address well-posedness issues of the system \eqref{eq1.1} in critical Besov spaces either
$\dot{B}^{-\alpha}_{\infty,1}(\mathbb{R}^{n})$ with $1<\alpha<2$ or $\dot{B}^{-1}_{\infty,1}(\mathbb{R}^{n})$  with $\alpha=1$. These critical spaces are marginal cases adapted to the system \eqref{eq1.1}.  The second  novelty of this paper is that we use the Gevrey class regularity to prove analyticity of the system \eqref{eq1.1}. The choice of this argument is motivated by the work of  Foias and Temam \cite{FT89} for estimating space analyticity radius of the Navier-Stokes equations (similar results were extended by many authors to various equations, see \cite{BBT12, BBT13, B14, BS07, HW13, L00} for more details). Our result characterizes space analyticity radius of solutions  and has an important physical interpretation: at this length scale the viscous effects and the nonlinear
inertial effects are roughly comparable, below this length scale the Fourier spectrum decays exponentially.
As a consequence of analyticity result, we obtain temporal decay rates of higher order Besov norms of the solutions.

Now we state the main results of this paper as follows.

\begin{theorem}\label{th1.1} Let $n\geq 2$, $1<\alpha\leq 2$. Assume that $u_{0}\in\dot{B}^{-\alpha+\frac{n}{p}}_{p,q}(\mathbb{R}^{n})$ with $1\leq p,q\leq\infty$. Then we have the following results:
\begin{itemize}
 \item [(i)] (Well-posedness for $1\leq p<\infty$)  Let $1\leq p<\infty$. Then there exists a $T^{*}=T^{*}(u_{0})>0$ such that the system \eqref{eq1.1} has a unique solution  $u\in\mathcal{X}_{T^{*}}$, where
     \begin{equation}\label{eq1.4}
     \mathcal{X}_{T^{*}}:=\widetilde{L}^{\infty}(0,T^{*}; \dot{B}^{-\alpha+\frac{n}{p}}_{p,q}(\mathbb{R}^{n}))\cap\widetilde{L}^{\rho_{1}}(0,T^{*}; \dot{B}^{s_1}_{p,q}(\mathbb{R}^{n}))
     \cap\widetilde{L}^{\rho_{2}}(0,T^{*}; \dot{B}^{s_2}_{p,q}(\mathbb{R}^{n}))
     \end{equation}
     with
     \begin{equation*}
     s_{1}=-1+\frac{n}{p}+\varepsilon,\ \ s_{2}=-1+\frac{n}{p}-\varepsilon, \ \ \rho_{1}=\frac{\alpha}{\alpha-1+\varepsilon},
     \ \ \rho_{2}=\frac{\alpha}{\alpha-1-\varepsilon}, \ \ 0<\varepsilon<\alpha-1.
     \end{equation*}
     If $T^{*}<\infty$, then
     \begin{equation*}
     \|u\|_{\widetilde{L}^{\rho_{1}}_{T^{*}}(\dot{B}^{s_1}_{p,q})
     \cap\widetilde{L}^{\rho_{2}}_{T^{*}}(\dot{B}^{s_2}_{p,q})}=\infty.
     \end{equation*}
     Moreover, if  $u_0\in\dot{B}^{-\alpha+\frac{n}{p}}_{p,q}(\mathbb{R}^{n})$ is sufficiently small, then $T^{*}=\infty$.

 \item [(ii)] (Analyticity for $1<p<\infty$)  Let  $1<p<\infty$. Then the solution obtained in (i) satisfying
   \begin{equation}\label{eq1.5}
     e^{t^{\frac{1}{\alpha}}\Lambda_{1}}u\in \mathcal{X}_{T^{*}},
 \end{equation}
 where the operator $\Lambda_{1}$ is the Fourier multiplier whose symbol is given by $|\xi|_{1}=|\xi_{1}|+\cdots+|\xi_{n}|$.
 Moreover, if  $u_0\in\dot{B}^{-\alpha+\frac{n}{p}}_{p,q}(\mathbb{R}^{n})$ is sufficiently small, then $T^{*}=\infty$.

  \item[(iii)] (Well-posedness for $p=\infty$)  Let $1<\alpha<2$ and $p=\infty$, suppose that $\|u_0\|_{\dot{B}^{-\alpha}_{\infty,1}}$ is sufficiently small. Then the system \eqref{eq1.1} has a unique solution $u$ satisfying
       \begin{equation}\label{eq1.6}
     u\in\widetilde{L}^{\infty}(0, \infty; \dot{B}^{-\alpha}_{\infty,1}(\mathbb{R}^{n}))\cap\widetilde{L}^{1}(0, \infty; \dot{B}^{0}_{\infty,1}(\mathbb{R}^{n})).
     \end{equation}

 \item [(iv)] (Analyticity for $p=\infty$)  Let $1<\alpha<2$ and $p=\infty$. Then the solution obtained in (iii) satisfying
  \begin{equation}\label{eq1.7}
   e^{t^{\frac{1}{\alpha}}\Lambda_{1}}u\in \widetilde{L}^{\infty}(0, \infty; \dot{B}^{-\alpha}_{\infty,1}(\mathbb{R}^{n}))\cap\widetilde{L}^{1}(0, \infty; \dot{B}^{0}_{\infty,1}(\mathbb{R}^{n})).
    \end{equation}

 \item [(v)](Decay rate for $1<p\leq\infty$) For any $\sigma\geq0$, $1<p<\infty$ or $p=\infty$ and $q=1$, the global solution obtained in (i)  and (iii) satisfying
\begin{equation}\label{eq1.8}
  \|\Lambda^{\sigma}u(t)\|_{\dot{B}^{-\alpha+\frac{n}{p}}_{p,q}}
  \leq C_{\sigma}t^{-\frac{\sigma}{\alpha}} \|u_{0}\|_{\dot{B}^{-\alpha+\frac{n}{p}}_{p,q}},
\end{equation}
where $C_{\sigma}:=\|\Lambda^{\sigma}e^{-\Lambda_{1}}\|_{L^{1}}$.
\end{itemize}
\end{theorem}

\noindent\textbf{Remark 1.1} We mention here that Bourgain and Pavlovi\'{c} \cite{BP08} proved ill-posedness for the 3D Navier-Stokes equations in $\dot{B}^{-1}_{\infty,\infty}(\mathbb{R}^{3})$. Subsequently, Yoneda \cite{Y10} proved ill-posedness in some function spaces, which are larger than $\dot{B}^{-1}_{\infty,2}(\mathbb{R}^{3})$ but smaller than $\dot{B}^{-1}_{\infty,q}(\mathbb{R}^{3})$ with $2<q\leq\infty$; Wang \cite{W15} finally proved ill-posedness for the 3D Navier-Stokes equations in $\dot{B}^{-1}_{\infty,q}(\mathbb{R}^{3})$ for all $1\leq q\leq2$. Note that  when $\alpha=2$, $\dot{B}^{-1}_{\infty,q}(\mathbb{R}^{n})$ for the Navier-Stokes equations corresponds to $\dot{B}^{-2}_{\infty,q}(\mathbb{R}^{n})$ for the system \eqref{eq1.1}, therefore, we cannot expect the well-posedness of the system \eqref{eq1.1} in $\dot{B}^{-2}_{\infty,q}(\mathbb{R}^{n})$ for $1\leq q\leq\infty$. However, when $1<\alpha<2$, Theorem \ref{th1.1} shows that the system \eqref{eq1.1} is well-posedness in $\dot{B}^{-\alpha}_{\infty,1}(\mathbb{R}^{n})$.

\noindent\textbf{Remark 1.2}   We emphasize here that the exponential operator $e^{t^{\frac{1}{\alpha}}\Lambda_{1}}$ is quantified by the operator $\Lambda_{1}$, whose symbol is given by the $l^{1}$ norm $|\xi|_{1}=\sum_{i=1}^{n}|\xi_{i}|$, rather than the usual operator $\Lambda=\sqrt{-\Delta}$, whose symbol is given by the $l^{2}$ norm $|\xi|=(\sum_{i=1}^{n}|\xi_{i}|^{2})^{\frac{1}{2}}$. This approach enables us to avoid
cumbersome recursive estimation of higher order derivatives and intricate combinatorial arguments to get the desired decay estimates of solutions, see \cite{Y11, ZLC11}.

\noindent\textbf{Remark 1.3} The method we use to prove well-posedness of \eqref{eq1.1} in critical Besov space  $\dot{B}^{-\alpha+\frac{n}{p}}_{p,q}(\mathbb{R}^{n})$  is the Chemin mono-norm method, which is different from the methods used in \cite{I11} and \cite{Z10}.

Corresponding to Theorem \ref{th1.1},  in the case $\alpha=1$, we obtain the following results.

\begin{theorem}\label{th1.2} Let $n\geq 2$, $\alpha=1$. Assume that $u_0\in\dot{B}^{-1+\frac{n}{p}}_{p,1}(\mathbb{R}^{n})$ with $1\leq p\leq\infty$. Then we have the following results:
\begin{itemize}
 \item [(i)] (Well-posedness for $1\leq p<\infty$)  Let $1\leq p<\infty$, suppose that $\|u_0\|_{\dot{B}^{-1+\frac{n}{p}}_{p,1}}$ is sufficiently small. Then the system \eqref{eq1.1} has a unique solution $u$ satisfying
     \begin{equation}\label{eq1.9}
     u\in\widetilde{L}^{\infty}(0, \infty; \dot{B}^{-1+\frac{n}{p}}_{p,1}(\mathbb{R}^{n})).
     \end{equation}
 \item [(ii)] (Analyticity for $1<p<\infty$)  Let $1<p<\infty$. Then the solution obtained in (i) satisfying
   \begin{equation}\label{eq1.10}
     e^{t^{\frac{1}{2n}}\Lambda_{1}}u\in \widetilde{L}^{\infty}(0, \infty; \dot{B}^{-1+\frac{n}{p}}_{p,1}(\mathbb{R}^{n})).
 \end{equation}
 \item[(iii)] (Well-posedness for $p=\infty$)  Let  $p=\infty$, suppose that $\|u_0\|_{\dot{B}^{-1}_{\infty,1}}$ is sufficiently small. Then the system \eqref{eq1.1} has a unique solution $u$ satisfying
 \begin{equation}\label{eq1.11}
     u\in\widetilde{L}^{\infty}(0, \infty; \dot{B}^{-1}_{\infty,1}(\mathbb{R}^{n}))\cap\widetilde{L}^{1}(0, \infty; \dot{B}^{0}_{\infty,1}(\mathbb{R}^{n})).
 \end{equation}
 \item [(iv)] (Analyticity for $p=\infty$)  Let $p=\infty$. Then the solution obtained in (iii) satisfying
\begin{equation}\label{eq1.12}
   e^{t^{\frac{1}{2n}}\Lambda_{1}}u\in \widetilde{L}^{\infty}(0, \infty; \dot{B}^{-1}_{\infty,1}(\mathbb{R}^{n}))\cap\widetilde{L}^{1}(0, \infty; \dot{B}^{0}_{\infty,1}(\mathbb{R}^{n})).
 \end{equation}
 \item [(v)](Decay rate for $1<p\leq\infty$) For any $\sigma\geq0$ and $1<p\leq\infty$, the global solution obtained in (i) and (iii) satisfying
\begin{equation}\label{eq1.13}
  \|\Lambda^{\sigma}u(t)\|_{\dot{B}^{-1+\frac{n}{p}}_{p,1}}
  \leq \widetilde{C}_{\sigma}t^{-\sigma} \|u_{0}\|_{\dot{B}^{-1+\frac{n}{p}}_{p,1}},
\end{equation}
where $\widetilde{C}_{\sigma}:=\|\Lambda^{\sigma}e^{-\frac{1}{2n}\Lambda_{1}}\|_{L^{1}}$.
\end{itemize}
\end{theorem}
\noindent\textbf{Remark 1.4} In the case $\alpha=1$,  since the dissipative operator $e^{-\frac{1}{2}t\Lambda}$ is not strong enough to dominate the operator $e^{t\Lambda_{1}}$, we need to define Gevrey operator more carefully. Notice the fact that $\frac{1}{2n}|\xi|_{1}<\frac{1}{2}|\xi|$ for all  $\xi\in\mathbb{R}^{n}$. Thus the Gevrey operator can be defined by $e^{\frac{1}{2n}t\Lambda_{1}}u$.

Before ending this section, let us sketch, for example, the proof of analyticity part in Theorem \ref{th1.1}. Setting  $U(t)=e^{t^{\frac{1}{\alpha}}\Lambda_{1}}u(t)$. Then
we see that $U(t)$ satisfies the following integral equation:
\begin{equation*}
   U(t)=e^{t^{\frac{1}{\alpha}}\Lambda_{1}-t\Lambda^{\alpha}}u_{0}
  -\int_{0}^{t}e^{[(t^{\frac{1}{\alpha}}-\tau^{\frac{1}{\alpha}})\Lambda_{1}-(t-\tau)\Lambda^{\alpha}]}\nabla\cdot e^{\tau^{\frac{1}{\alpha}}\Lambda_{1}}\left(e^{-\tau^{\frac{1}{\alpha}}\Lambda_{1}}U(\tau)
  e^{-\tau^{\frac{1}{\alpha}}\Lambda_{1}}\nabla(-\Delta)^{-1}U(\tau)\right)d\tau.
\end{equation*}
Note that since $e^{t^{\frac{1}{\alpha}}|\xi|_{1}}$ can be dominated by $e^{-t|\xi|^{\alpha}}$ if $|\xi|$ is sufficiently large, the behavior
of the linear term $e^{t^{\frac{1}{\alpha}}\Lambda_{1}-t\Lambda^{\alpha}}u_{0}$ closely resemble that of
$e^{-t\Lambda^{\alpha}}u_{0}$.
In order to tackle with the nonlinear term, we resort to \cite{L02} and \cite{BBT13} to find out the nice boundedness property of the following bilinear operator:
\begin{align*}
  \mathcal{B}_{t}(f,g)&:=e^{t^{\frac{1}{\alpha}}\Lambda_{1}}(e^{-t^{\frac{1}{\alpha}}\Lambda_{1}}fe^{-t^{\frac{1}{\alpha}}\Lambda_{1}}g)\nonumber\\
  &=\frac{1}{(2\pi)^{n}}\int_{\mathbb{R}^{n}}\int_{\mathbb{R}^{n}}
  e^{ix\cdot(\xi+\eta)}e^{t^{\frac{1}{\alpha}}(|\xi+\eta|_{1}-|\xi|_{1}-|\eta|_{1})}\hat{f}(\xi)\hat{g}(\xi)d\xi d\eta.
\end{align*}
Based on the desired properties of $\mathcal{B}_{t}(f,g)$, we can modify the argument of the proof of well-posedness results in Theorem \ref{th1.1} to obtain Gevrey regularity.

This paper is organized as follows: We shall collect some basic facts on Littlewood-Paley dyadic decomposition theory and various product laws in Besov spaces in Section 2, then  prove
 Theorem \ref{th1.1} in Section 3, and Theorem \ref{th1.2} in Section 4.

\section{Preliminaries}

\subsection{Notations}

In this paper, we shall use the following notations.
\begin{itemize}
\item For two
constants $A$ and $B$, the notation $A\lesssim B$ means that there
is a uniform constant $C$ (always independent of $x, t$), which may vary from line to line, such that $A\le CB$.

\item For $x=(x_{1},\cdots, x_{n})\in\mathbb{R}^{n}$, we denote
$|x|_{p}=(|x_{1}|^{p}+\cdots+|x_{n}|^{p})^{\frac{1}{p}}$ and $|x|=|x|_{2}$.

\item The operator $\Lambda_{1}$ is the Fourier multiplier whose symbol is given by $|\xi|_{1}=|\xi_{1}|+\cdots+|\xi_{n}|$.

\item For a quasi-Banach space $X$ and for any $0<T\leq\infty$, we use standard notation $L^{p}(0,T; X)$ or $L^{p}_{T}(X)$  for the quasi-Banach space of Bochner measurable functions
$f$ from $(0, T)$ to $X$ endowed with the norm
\begin{equation*}
\|f\|_{L^{p}_{T}(X)}:=
\begin{cases}
  (\int_{0}^{T}\|f(\cdot,t)\|_{X}^{p}dt)^{\frac{1}{p}}\ \ \ &\text{for}\ \ \ 1\leq p<\infty,\\
  \sup_{0\leq t\leq T}\|f(\cdot,t)\|_{X}\ \ \ &\text{for}\ \ \ p=\infty.
\end{cases}
\end{equation*}
In particular, if $T=\infty$, we use $\|f\|_{L^{p}_{t}(X)}$ instead of $\|f\|_{L^{p}_{\infty}(X)}$.

\item For any function space $X$
and the operator $\mathcal{T}: X\rightarrow X$, we denote
$$
  \mathcal{T}X:=\{Tf:\ f\in X\} \ \ \text{and}\ \ \|f\|_{\mathcal{T}X}:=\|\mathcal{T}f\|_{X}.
$$

\item The linear space of all multipliers on $L^{p}$ is denoted by $\mathcal{M}_{p}$ and the norm on which is defined by
$$
  \|f\|_{\mathcal{M}_{p}}:=\sup\{\|\mathcal{F}^{-1}[f\mathcal{F}g]\|_{L^{p}}:\ \ \forall g\in\mathcal{S}(\mathbb{R}^{n}), \ \|g\|_{L^{p}}=1\}.
$$
\end{itemize}

\subsection{Littlewood-Paley theory and Besov spaces}

The proof of Theorems \ref{th1.1} and \ref{th1.2} are formulated by the dyadic decomposition in the Littlewood-Paley theory.
 Let us briefly explain how it may be built in $\mathbb{R}^{n}$. Let
$\mathcal{S}(\mathbb{R}^{n})$ be the Schwartz class of rapidly
decreasing function, and $\mathcal{S}'(\mathbb{R}^{n})$ of temperate distributions be the dual set of $\mathcal{S}(\mathbb{R}^{n})$.
Let $\varphi\in\mathcal{S}(\mathbb{R}^{n})$ be a  smooth radial function valued in $[0,1]$ such that  $\varphi$ is supported in the shell $\mathcal{C}=\{\xi\in\mathbb{R}^{n},\ \frac{3}{4}\leq
|\xi|\leq\frac{8}{3}\}$, and
\begin{align*}
  \sum_{j\in\mathbb{Z}}\varphi(2^{-j}\xi)=1, \ \ \ \forall\xi\in\mathbb{R}^{n}\backslash\{0\}.
\end{align*}
Then for any $f\in\mathcal{S}'(\mathbb{R}^{n})$, we set
for all $j\in \mathbb{Z}$,
\begin{align}\label{eq2.1}
  \Delta_{j}f:=\varphi(2^{-j}D)f \ \ \ \text{and}\ \ \
  S_{j}f:=\sum_{k\leq j-1}\Delta_{k}f.
\end{align}
By telescoping the series, we have the following homogeneous Littlewood-Paley decomposition:
\begin{equation*}
  f=\sum_{j\in\mathbb{Z}}\Delta_{j}f \ \ \text{for}\ \
  f\in\mathcal{S}'(\mathbb{R}^{n})/\mathcal{P}(\mathbb{R}^{n}),
\end{equation*}
where $\mathcal{P}(\mathbb{R}^{n})$ is the set of polynomials (see \cite{BCD11}).
We remark here that the Littlewood-Paley decomposition satisfies the property of almost orthogonality, that is to say, for any $f, g\in\mathcal{S}'(\mathbb{R}^{n})/\mathcal{P}(\mathbb{R}^{n})$, the following properties hold:
\begin{align}\label{eq2.2}
  \Delta_{i}\Delta_{j}f\equiv0\ \ \ \text{if}\ \ \ |i-j|\geq 2\ \ \ \text{and}\ \ \
  \Delta_{i}(S_{j-1}f\Delta_{j}g)\equiv0 \ \ \ \text{if}\ \ \ |i-j|\geq 5.
\end{align}

Using the above decomposition, the  stationary/time dependent homogeneous Besov space can be defined as follows:

\begin{definition}\label{de2.1}
Let $s\in \mathbb{R}$, $1\leq p,q\leq\infty$ and $f\in\mathcal{S}'(\mathbb{R}^{n})$, we set
\begin{equation*}
  \|f\|_{\dot{B}^{s}_{p,q}}:= \begin{cases} \left(\sum_{j\in\mathbb{Z}}2^{jsq}\|\Delta_{j}f\|_{L^{p}}^{q}\right)^{\frac{1}{q}}
  \ \ &\text{for}\ \ 1\leq q<\infty,\\
  \sup_{j\in\mathbb{Z}}2^{js}\|\Delta_{j}f\|_{L^{p}}\ \
  &\text{for}\ \
  q=\infty.
 \end{cases}
\end{equation*}
Then the homogeneous Besov
space $\dot{B}^{s}_{p,q}(\mathbb{R}^{n})$ is defined by
\begin{itemize}
\item For $s<\frac{n}{p}$ (or $s=\frac{n}{p}$ if $q=1$), we define
\begin{equation*}
  \dot{B}^{s}_{p,q}(\mathbb{R}^{n}):=\Big\{f\in \mathcal{S}'(\mathbb{R}^{n}):\ \
  \|f\|_{\dot{B}^{s}_{p,q}}<\infty\Big\}.
\end{equation*}
\item If $k\in\mathbb{N}$ and $\frac{n}{p}+k\leq s<\frac{n}{p}+k+1$ (or $s=\frac{n}{p}+k+1$ if $q=1$), then $\dot{B}^{s}_{p,q}(\mathbb{R}^{n})$
is defined as the subset of distributions $f\in\mathcal{S}'(\mathbb{R}^{n})$ such that $\partial^{\beta}f\in\mathcal{S}'(\mathbb{R}^{n})$
whenever $|\beta|=k$.
\end{itemize}
\end{definition}

\begin{definition}\label{de2.2} For $0<T\leq\infty$, $s\leq \frac{n}{p}$ (resp. $s\in \mathbb{R}$),
$1\leq p, q, \rho\leq\infty$. We define the mixed time-space $\widetilde{L}^{\rho}(0,T; \dot{B}^{s}_{p,q}(\mathbb{R}^{n}))$
as the completion of $\mathcal{C}([0,T]; \mathcal{S}(\mathbb{R}^{n}))$ by the norm
$$
  \|f\|_{\widetilde{L}^{\rho}_{T}(\dot{B}^{s}_{p,q})}:=\left(\sum_{j\in\mathbb{Z}}2^{jsq}\left(\int_{0}^{T}
  \|\Delta_{j}f(\cdot,t)\|_{L^{p}}^{\rho}dt\right)^{\frac{q}{\rho}}\right)^{\frac{1}{q}}<\infty
$$
with the usual change if $\rho=\infty$
or $q=\infty$.  For simplicity, we use $\|f\|_{\widetilde{L}^{\rho}_{t}(\dot{B}^{s}_{p,q})}$ instead of $\|f\|_{\widetilde{L}^{\rho}_{\infty}(\dot{B}^{s}_{p,q})}$.
\end{definition}

In what follows, we shall frequently use the following Bony's homogeneous paraproduct decomposition, which is a mathematical tool
to define a generalized product between two temperate
distributions (see \cite{B81}). Let $f$ and $g$ be two temperate distributions, the
paraproduct between $f$ and $g$ is defined by
\begin{equation*}
  T_{f}g:=\sum_{j\in\mathbb{Z}}S_{j-1}f\Delta_{j}g=\sum_{j\in\mathbb{Z}}\sum_{k\leq j-2}\Delta_{k}f\Delta_{j}g.
\end{equation*}
Formally,  we have the following Bony's decomposition:
\begin{equation*}
  fg=T_{f}g+T_{g}f+R(f,g),
\end{equation*}
where
\begin{equation*}
  R(f,g):=\sum_{j\in\mathbb{Z}}\sum_{|j-j'|\leq 1}\Delta_{j}f\Delta_{j'}g.
\end{equation*}

\subsection{Essential lemmas}

For the convenience of the reader, we recall some basic facts of the Littlewood-Paley theory,  one may refer to \cite{BCD11}, \cite{L02} for more details.

\begin{lemma}\label{le2.3}{\em (\cite{BCD11}, \cite{L02})}
Let $\mathcal{B}$ be a ball, and $\mathcal{C}$ be a ring in
$\mathbb{R}^{n}$. There exists a constant $C$ such that for any
positive real number $\lambda$, any nonnegative integer $k$ and any
couple of real numbers $(p,r)$ with $1\leq p\leq r\leq \infty$, we
have
\begin{equation}\label{eq2.3}
   \operatorname{supp}\mathcal{F}(f)\subset\lambda\mathcal{B}\ \ \Rightarrow\ \   \sup_{|\alpha|=k}\|\partial^{\alpha}f\|_{L^{r}}\leq
   C^{k+1}\lambda^{k+n(\frac{1}{p}-\frac{1}{r})}\|f\|_{L^{p}},
\end{equation}
\begin{equation}\label{eq2.4}
   \operatorname{supp}\mathcal{F}(f)\subset\lambda\mathcal{C} \ \ \Rightarrow\ \   C^{-1-k}\lambda^{k}\|f\|_{L^{p}}\leq
   \sup_{|\alpha|=k}\|\partial^{\alpha}f\|_{L^{p}}\leq  C^{1+k}\lambda^{k}\|f\|_{L^{p}}.
\end{equation}
\end{lemma}

\begin{lemma}\label{le2.4} {\em (\cite{BCD11}, \cite{L02})}
Let $f$ be a smooth function on $\mathbb{R}^{n}\backslash\{0\}$ which is homogeneous of degree $m$. Then for any $s\in\mathbb{R}$,
$1\leq p, q\leq \infty$, and
\begin{equation*}
  s-m<\frac{n}{p}, \ \ \ \text{or}\ \ \ s-m=\frac{n}{p} \ \ \ \text{and}\ \ \  q=1,
\end{equation*}
the operator $f(D)$ is continuous from $\dot{B}^{s}_{p,q}(\mathbb{R}^{n})$ to $\dot{B}^{s-m}_{p,q}(\mathbb{R}^{n})$.
\end{lemma}

\begin{lemma}\label{le2.5} {\em (\cite{WY08})}
Let $\mathcal{C}$ be a ring in $\mathbb{R}^{n}$. There exist two
positive constants $\kappa$ and $\mathcal{K}$ such that for any $p\in[1,\infty]$
and any couple $(t,\lambda)$ of positive real numbers, we have
\begin{equation}\label{eq2.5}
   \operatorname{supp}\mathcal{F}(f)\subset\lambda\mathcal{C} \ \ \Rightarrow\ \
   \|e^{t\Lambda^{\alpha}}f\|_{L^{p}}\leq \mathcal{K}e^{-\kappa\lambda^{\alpha} t}\|f\|_{L^{p}}.
\end{equation}
%----(eq2.7)--
\end{lemma}

\section{The case  $1<\alpha\leq 2$:  Proof of Theorem \ref{th1.1}}
In this section we prove Gevrey analyticity of the system \eqref{eq1.1} in critical Besov space $\dot{B}^{-\alpha+\frac{n}{p}}_{p,q}(\mathbb{R}^{n})$ with $1<\alpha\leq 2$, $1<p\leq\infty$ and $1\leq q\leq \infty$. The proof is based on an adequate modification of the proof of
local in time existence with any initial data and global in time existence with small initial data to the system \eqref{eq1.1}, thus  we begin with the detailed proof of the first part of
Theorem
\ref{th1.1}.

\subsection{The case $1\leq p<\infty$: Well-posedness}
In this subsection, we intend to establish local well-posedness with any initial data and global well-posedness with small initial data to the system \eqref{eq1.1} in critical Besov space $\dot{B}^{-\alpha+\frac{n}{p}}_{p,q}(\mathbb{R}^{n})$ for $1\leq p<\infty$.  Firstly we are concerned with the Cauchy problem of the fractional power dissipative equation:
\begin{equation}\label{eq3.1}
\begin{cases}
  \partial_{t}u+\Lambda^{\alpha} u= f, \ \
  &x\in\mathbb{R}^{n}, \ t>0,\\
  u(x,0)=u_{0}(x), \ \ &x\in\mathbb{R}^{n}.
\end{cases}
\end{equation}
\begin{proposition}\label{pro3.1} {\em (\cite{BW09})}
Let $s\in \mathbb{R}$, $1\leq p,q,\rho_1\leq\infty$ and
$0<T\leq\infty$. Assume that $u_{0}\in
\dot{B}^{s}_{p,q}(\mathbb{R}^{n})$ and $f\in\widetilde{L}^{\rho_1}_{T}(\dot{B}^{s+\frac{\alpha}{\rho_{1}}-\alpha}_{p,q}(\mathbb{R}^{n}))$. Then \eqref{eq3.1} has
a unique solution $u\in\underset{\rho_1\leq
\rho\leq\infty}{\cap}\widetilde{L}^{\rho}_{T}(\dot{B}^{s+\frac{\alpha}{\rho}}_{p,q}(\mathbb{R}^{n}))$. In addition, there exists a
constant $C>0$ depending only on $\alpha $ and $n$ such that for any $\rho_1\leq
\rho\leq\infty$, we have
\begin{equation}\label{eq3.2}
  \|u\|_{\widetilde{L}^{\rho}_{T}(\dot{B}^{s+\frac{\alpha}{\rho}}_{p,q})}\leq
  C\big(\|u_{0}\|_{\dot{B}^{s}_{p,q}}+\|f\|_{\widetilde{L}^{\rho_1}_{T}(\dot{B}^{s+\frac{\alpha}{\rho_{1}}-\alpha}_{p,q})}\big).
\end{equation}
In particular,  if $f\in\widetilde{L}^{1}_{T}(\dot{B}^{s}_{p,q}(\mathbb{R}^{n}))$, then we have
\begin{equation}\label{eq3.3}
  \|u\|_{\widetilde{L}^{\infty}_{T}(\dot{B}^{s}_{p,q})\cap\widetilde{L}^{1}_{T}(\dot{B}^{s+\alpha}_{p,q})}\leq
  C\big(\|u_{0}\|_{\dot{B}^{s}_{p,q}}+\|f\|_{\widetilde{L}^{1}_{T}(\dot{B}^{s}_{p,q})}\big).
\end{equation}
\end{proposition}

Next, by using in a fundamental way the algebraical structure of  the system \eqref{eq1.1}, we establish the following crucial bilinear estimates in time dependent Besov spaces.

\begin{lemma}\label{le3.2} Let $s>0$, $1\leq p<\infty$, $1\leq q, \rho, \rho_{1}, \rho_{2}\leq \infty$ with $\frac{1}{\rho}=\frac{1}{\rho_{1}}+\frac{1}{\rho_{2}}$. Then for any $\varepsilon>0$,  $0<T\leq\infty$, we have
\begin{align}\label{eq3.4}
     \|u\nabla(-\Delta)^{-1}v+v\nabla(-\Delta)^{-1}u\|_{\widetilde{L}^{\rho}_{T}(\dot{B}^{s}_{p,q})}&\lesssim
     \|u\|_{\widetilde{L}^{\rho_{1}}_{T}(\dot{B}^{s+\varepsilon}_{p,q})}\|v\|_{\widetilde{L}^{\rho_{2}}_{T}(\dot{B}^{-1+\frac{n}{p}-\varepsilon}_{p,q})}\nonumber\\
     &\ \  +\|u\|_{\widetilde{L}^{\rho_{2}}_{T}(\dot{B}^{-1+\frac{n}{p}-\varepsilon}_{p,q})}\|v\|_{\widetilde{L}^{\rho_{1}}_{T}(\dot{B}^{s+\varepsilon}_{p,q})}.
\end{align}
Moreover, if we choose $\varepsilon=0$, then \eqref{eq3.4} also holds for $q=1$.
\end{lemma}
\begin{proof}
Thanks to Bony's paraproduct decomposition, we have
\begin{equation}\label{eq3.5}
    u\nabla(-\Delta)^{-1}v+v\nabla(-\Delta)^{-1}u:=I_{1}+I_{2}+I_{3},
\end{equation}
where
\begin{align*}
    I_{1}:&=\sum_{j'\in\mathbb{Z}}\Delta_{j'}u\nabla(-\Delta)^{-1}S_{j'-1}v+\Delta_{j'}v\nabla(-\Delta)^{-1}S_{j'-1}u;\\
    I_{2}:&=\sum_{j'\in\mathbb{Z}}S_{j'-1}u\nabla(-\Delta)^{-1}\Delta_{j'}v+S_{j'-1}v\nabla(-\Delta)^{-1}\Delta_{j'}u;\\
    I_{3}:&=\sum_{j'\in\mathbb{Z}}\sum_{|j'-j''|\leq1}\Delta_{j'}u\nabla(-\Delta)^{-1}\Delta_{j''}v+\Delta_{j'}v\nabla(-\Delta)^{-1}\Delta_{j''}u.
\end{align*}
In the sequel, we estimate $I_{i}$ ($i=1,2,3$) one by one.  For $I_{1}$, we need only to deal with the first term  $\sum_{j'\in\mathbb{Z}}\Delta_{j'}u\nabla(-\Delta)^{-1}S_{j'-1}v$, while the second one can be done analogously,  thus using the facts \eqref{eq2.1} and \eqref{eq2.2},  and applying H\"{o}lder's inequality and Lemmas \ref{le2.3} and \ref{le2.4}, one has
\begin{align}\label{eq3.6}
  \|\Delta_{j}\sum_{j'\in\mathbb{Z}}\Delta_{j'}u\nabla(-\Delta)^{-1}S_{j'-1}v\|_{L^{\rho}_{T}(L^{p})}
 & \lesssim \sum_{|j'-j|\leq 4}\|\Delta_{j'}u\|_{L^{\rho_{1}}_{T}(L^{p})}\|\nabla(-\Delta)^{-1}S_{j'-1}v\|_{L^{\rho_{2}}_{T}(L^{\infty})}\nonumber\\
 & \lesssim \sum_{|j'-j|\leq 4}\|\Delta_{j'}u\|_{L^{\rho_{1}}_{T}(L^{p})}\sum_{k\leq j'-2}2^{(-1+\frac{n}{p})k}\|\Delta_{k}v\|_{L^{\rho_{2}}_{T}(L^{p})}\nonumber\\
  &\lesssim \sum_{|j'-j|\leq 4}\|\Delta_{j'}u\|_{L^{\rho_{1}}_{T}(L^{p})}\sum_{k\leq j'-2}2^{\varepsilon k}2^{(-1+\frac{n}{p}-\varepsilon)k}\|\Delta_{k}v\|_{L^{\rho_{2}}_{T}(L^{p})}\nonumber\\
  &\lesssim \sum_{|j'-j|\leq 4}2^{-sj'}2^{(s+\varepsilon)j'}\|\Delta_{j'}u\|_{L^{\rho_{1}}_{T}(L^{p})}\|v\|_{\widetilde{L}^{\rho_{2}}_{T}(\dot{B}^{-1+\frac{n}{p}-\varepsilon}_{p,q})}.
\end{align}
Multiplying \eqref{eq3.6}  by $2^{sj}$, then taking $l^{q}$ norm to the resulting inequality, we obtain
\begin{align*}
      \|\sum_{j'\in\mathbb{Z}}\Delta_{j'}u\nabla(-\Delta)^{-1}S_{j'-1}v\|_{\widetilde{L}^{\rho}_{T}(\dot{B}^{s}_{p,q})}\lesssim
       \|u\|_{\widetilde{L}^{\rho_{1}}_{T}(\dot{B}^{s+\varepsilon}_{p,q})}\|v\|_{\widetilde{L}^{\rho_{2}}_{T}(\dot{B}^{-1+\frac{n}{p}-\varepsilon}_{p,q})},
\end{align*}
which implies that
\begin{align}\label{eq3.7}
  \|I_{1}\|_{\widetilde{L}^{\rho}_{T}(\dot{B}^{s}_{p,q})}\lesssim
       \|u\|_{\widetilde{L}^{\rho_{1}}_{T}(\dot{B}^{s+\varepsilon}_{p,q})}\|v\|_{\widetilde{L}^{\rho_{2}}_{T}(\dot{B}^{-1+\frac{n}{p}-\varepsilon}_{p,q})}
       +\|u\|_{\widetilde{L}^{\rho_{2}}_{T}(\dot{B}^{-1+\frac{n}{p}-\varepsilon}_{p,q})}\|v\|_{\widetilde{L}^{\rho_{1}}_{T}(\dot{B}^{s+\varepsilon}_{p,q})}.
\end{align}
Similarly, for the first term of $I_{2}$,  applying H\"{o}lder's inequality and Lemmas
\ref{le2.3} and \ref{le2.4} again, we see that
\begin{align}\label{eq3.8}
   \|\Delta_{j}\sum_{j'\in\mathbb{Z}}S_{j'-1}u\nabla(-\Delta)^{-1}\Delta_{j'}v\|_{L^{\rho}_{T}(L^{p})}
   & \lesssim \sum_{|j'-j|\leq 4}\sum_{k\leq j'-2}2^{\frac{n}{p}k}\|\Delta_{k}u\|_{L^{\rho_{2}}_{T}(L^{p})}2^{-j'}\|\Delta_{j'}v\|_{L^{\rho_{1}}_{T}(L^{p})}\nonumber\\
  & \lesssim \sum_{|j'-j|\leq 4}\sum_{k\leq j'-2}2^{(1+\varepsilon)k}2^{(-1+\frac{n}{p}-\varepsilon)k}\|\Delta_{k}u\|_{L^{\rho_{2}}_{T}(L^{p})}2^{-j'}\|\Delta_{j'}v\|_{L^{\rho_{1}}_{T}(L^{p})}\nonumber\\
  &\lesssim \sum_{|j'-j|\leq 4}2^{-sj'}2^{(s+\varepsilon)j'}\|\Delta_{j'}v\|_{L^{\rho_{1}}_{T}(L^{p})}\|u\|_{\widetilde{L}^{\rho_{2}}_{T}(\dot{B}^{-1+\frac{n}{p}-\varepsilon}_{p,q})},
\end{align}
which yields directly that
\begin{align*}
  \|\sum_{j'\in\mathbb{Z}}S_{j'-1}u\nabla(-\Delta)^{-1}\Delta_{j'}v\|_{\widetilde{L}^{\rho}_{T}(\dot{B}^{s}_{p,q})}\lesssim
       \|u\|_{\widetilde{L}^{\rho_{2}}_{T}(\dot{B}^{-1+\frac{n}{p}-\varepsilon}_{p,q})}\|v\|_{\widetilde{L}^{\rho_{1}}_{T}(\dot{B}^{s+\varepsilon}_{p,q})}.
\end{align*}
Thus,
\begin{align}\label{eq3.9}
  \|I_{2}\|_{\widetilde{L}^{\rho}_{T}(\dot{B}^{s}_{p,q})}\lesssim
       \|u\|_{\widetilde{L}^{\rho_{1}}_{T}(\dot{B}^{s+\varepsilon}_{p,q})}\|v\|_{\widetilde{L}^{\rho_{2}}_{T}(\dot{B}^{-1+\frac{n}{p}-\varepsilon}_{p,q})}
       +\|u\|_{\widetilde{L}^{\rho_{2}}_{T}(\dot{B}^{-1+\frac{n}{p}-\varepsilon}_{p,q})}\|v\|_{\widetilde{L}^{\rho_{1}}_{T}(\dot{B}^{s+\varepsilon}_{p,q})}.
\end{align}
Now we tackle with the most difficult term $I_{3}$.  Based on careful analysis of the algebraical structure of the system \eqref{eq1.1},  we can split $I_{3}$ into the following three terms for $m=1,2,\cdots, n$:
\begin{equation}\label{eq3.10}
  I_{3}:=K_{1}+K_{2}+K_{3},
\end{equation}
where
\begin{align*}
    K_{1}:&=\sum_{j'\in\mathbb{Z}}\sum_{|j'-j''|\leq1}(-\Delta)\Big{\{}\big{(}(-\Delta)^{-1}\Delta_{j'}u\big{)}\big{(}\partial_{m}(-\Delta)^{-1}\Delta_{j''}v\big{)}\Big{\}};\\
    K_{2}:&=\sum_{j'\in\mathbb{Z}}\sum_{|j'-j''|\leq1}2\nabla\cdot\Big{\{}\big{(}(-\Delta)^{-1}\Delta_{j'}u\big{)}\big{(}\partial_{m}\nabla(-\Delta)^{-1}\Delta_{j''}v\big{)}\Big{\}};\\
    K_{3}:&=\sum_{j'\in\mathbb{Z}}\sum_{|j'-j''|\leq1}\partial_{m}\Big{\{}\big{(}(-\Delta)^{-1}\Delta_{j'}u\big{)}\Delta_{j''}v\Big{\}}.
\end{align*}
Moreover, since $K_{2}$ can be treated similarly to $K_{3}$, we treat $K_{1}$ and $K_{3}$
only. We first consider the case $2\leq p<\infty$, by using H\"{o}lder's inequality and Lemmas \ref{le2.3} and \ref{le2.4},  it follows from \eqref{eq2.2} that there exists $N_{0}\in\mathbb{N}$ such that
\begin{align}\label{eq3.11}
    \|\Delta_{j}K_{1}\|_{L^{\rho}_{T}(L^{p})}&\lesssim 2^{(2+\frac{n}{p})j}\sum_{j'\geq j-N_{0}}\sum_{|j'-j''|\leq1}
    \|(-\Delta)^{-1}\Delta_{j'}u\|_{L^{\rho_{1}}_{T}(L^{p})}\|\partial_{m}(-\Delta)^{-1}\Delta_{j''}v\|_{L^{\rho_{2}}_{T}(L^{p})}\nonumber\\
    &\lesssim 2^{(2+\frac{n}{p})j}\sum_{j'\geq j-N_{0}}\sum_{|j'-j''|\leq1}2^{(-2-s-\frac{n}{p})j'}2^{(s+\varepsilon)}
    \|\Delta_{j'}u\|_{L^{\rho_{1}}_{T}(L^{p})}2^{(-1+\frac{n}{p}-\varepsilon)j''}\|\Delta_{j''}v\|_{L^{\rho_{2}}_{T}(L^{p})}\nonumber\\
    &\lesssim 2^{-sj}\sum_{j'\geq j-N_{0}}2^{-(2+s+\frac{n}{p})(j'-j)}2^{(s+\varepsilon)j'}
    \|\Delta_{j'}u\|_{L^{\rho_{1}}_{T}(L^{p})}\|v\|_{L^{\rho_{2}}_{T}(\dot{B}^{-1+\frac{n}{p}-\varepsilon}_{p,q})}.
\end{align}
\begin{align}\label{eq3.12}
    \|\Delta_{j}K_{3}\|_{L^{\rho}_{T}(L^{p})}&\lesssim 2^{(1+\frac{n}{p})j}\sum_{j'\geq j-N_{0}}\sum_{|j'-j''|\leq1}
    \|(-\Delta)^{-1}\Delta_{j'}u\|_{L^{\rho_{1}}_{T}(L^{p})}\|\Delta_{j''}v\|_{L^{\rho_{2}}_{T}(L^{p})}\nonumber\\
   &\lesssim 2^{(1+\frac{n}{p})j}\sum_{j'\geq j-N_{0}}\sum_{|j'-j''|\leq1}2^{(-1-s-\frac{n}{p})j'}2^{(s+\varepsilon)j'}
    \|\Delta_{j'}u\|_{L^{\rho_{1}}_{T}(L^{p})}2^{(-1+\frac{n}{p}-\varepsilon)j''}\|\Delta_{j''}v\|_{L^{\rho_{2}}_{T}(L^{p})}\nonumber\\
    &\lesssim 2^{-sj}\sum_{j'\geq j-N_{0}}2^{-(1+s+\frac{n}{p})(j'-j)}2^{(s+\varepsilon)j'}
    \|\Delta_{j'}u\|_{L^{\rho_{1}}_{T}(L^{p})}\|v\|_{L^{\rho_{2}}_{T}(\dot{B}^{-1+\frac{n}{p}-\varepsilon}_{p,q})}.
\end{align}
On the other hand, in the case that $1\leq p<2$, we choose $2<p'\leq\infty$ such that $\frac{1}{p}+\frac{1}{p'}=1$, it follows that
\begin{align}\label{eq3.13}
    \|\Delta_{j}K_{1}\|_{L^{\rho}_{T}(L^{p})}&\lesssim 2^{(2+n-\frac{n}{p})j}\sum_{j'\geq j-N_{0}}\sum_{|j'-j''|\leq1}
    \|(-\Delta)^{-1}\Delta_{j'}u\|_{L^{\rho_{1}}_{T}(L^{p'})}\|\partial_{m}(-\Delta)^{-1}\Delta_{j''}v\|_{L^{\rho_{2}}_{T}(L^{p})}\nonumber\\
    &\lesssim 2^{(2+n-\frac{n}{p})j}\sum_{j'\geq j-N_{0}}\sum_{|j'-j''|\leq1}2^{(-2+n(\frac{1}{p}-\frac{1}{p'}))j'}
    \|\Delta_{j'}u\|_{L^{\rho_{1}}_{T}(L^{p})}2^{-j''}\|\Delta_{j''}v\|_{L^{\rho_{2}}_{T}(L^{p})}\nonumber\\
    &\lesssim 2^{(2+n-\frac{n}{p})j}\sum_{j'\geq j-N_{0}}2^{-(2+n+s-\frac{n}{p})j'}2^{(s+\varepsilon)j'}
    \|\Delta_{j'}u\|_{L^{\rho_{1}}_{T}(L^{p})}2^{(-1+\frac{n}{p}-\varepsilon)j'}\|\Delta_{j'}v\|_{L^{\rho_{2}}_{T}(L^{p})}\nonumber\\
    &\lesssim 2^{-sj}\sum_{j'\geq j-N_{0}}2^{-(2+n+s-\frac{n}{p})(j'-j)}2^{(s+\varepsilon)j'}
    \|\Delta_{j'}u\|_{L^{\rho_{1}}_{T}(L^{p})}\|v\|_{L^{\rho_{2}}_{T}(\dot{B}^{-1+\frac{n}{p}-\varepsilon}_{p,q})}.
\end{align}
\begin{align}\label{eq3.14}
    \|\Delta_{j}K_{3}\|_{L^{\rho}_{T}(L^{p})}&\lesssim 2^{(1+n-\frac{n}{p})j}\sum_{j'\geq j-N_{0}}\sum_{|j'-j''|\leq1}
    \|(-\Delta)^{-1}\Delta_{j'}u\|_{L^{\rho_{1}}_{T}(L^{p'})}\|\Delta_{j''}v\|_{L^{\rho_{2}}_{T}(L^{p})}\nonumber\\
    &\lesssim 2^{(1+n-\frac{n}{p})j}\sum_{j'\geq j-N_{0}}\sum_{|j'-j''|\leq1}2^{(-2+n(\frac{1}{p}-\frac{1}{p'}))j'}
    \|\Delta_{j'}u\|_{L^{\rho_{1}}_{T}(L^{p})}\|\Delta_{j''}v\|_{L^{\rho_{2}}_{T}(L^{p})}\nonumber\\
    &\lesssim 2^{(1+n-\frac{n}{p})j}\sum_{j'\geq j-N_{0}}2^{-(1+n+s-\frac{n}{p})j'}2^{(s+\varepsilon)j'}
    \|\Delta_{j'}u\|_{L^{\rho_{1}}_{T}(L^{p})}2^{(-1+\frac{n}{p}-\varepsilon)j'}\|\Delta_{j'}v\|_{L^{\rho_{2}}_{T}(L^{p})}\nonumber\\
   &\lesssim 2^{-sj}\sum_{j'\geq j-N_{0}}2^{-(1+n+s-\frac{n}{p})(j'-j)}2^{(s+\varepsilon)j'}
    \|\Delta_{j'}u\|_{L^{\rho_{1}}_{T}(L^{p})}\|v\|_{L^{\rho_{2}}_{T}(\dot{B}^{-1+\frac{n}{p}-\varepsilon}_{p,q})}.
\end{align}
Note that  under the hypotheses of Lemma \ref{le3.2},  we have
$$
    2+s+\frac{n}{p}>0,\  1+s+\frac{n}{p}>0, \ 2+n+s-\frac{n}{p}>0, \ 1+n+s-\frac{n}{p}>0.
$$
Then we infer from the estimates \eqref{eq3.11}--\eqref{eq3.14}  that  for all  $1\leq p<\infty$,
\begin{align}\label{eq3.15}
      \| I_{3}\|_{\widetilde{L}^{\rho}_{T}(\dot{B}^{s}_{p,q})}\lesssim
       \|u\|_{\widetilde{L}^{\rho_{1}}_{T}(\dot{B}^{s+\varepsilon}_{p,q})}\|v\|_{\widetilde{L}^{\rho_{2}}_{T}(\dot{B}^{-1+\frac{n}{p}-\varepsilon}_{p,q})}.
\end{align}
Hence, plugging \eqref{eq3.7}, \eqref{eq3.9} and \eqref{eq3.15} into \eqref{eq3.5}, we get \eqref{eq3.4}. We complete the proof of Lemma \ref{le3.2}.
\end{proof}

\medskip

Now we are in a position to prove well-posedness of the system \eqref{eq1.1} in the case that $1<\alpha\leq 2$ and $1\leq p<\infty$.
Define  the map
\begin{equation}\label{eq3.16}
    \mathbb{F}: \ u(t)\rightarrow e^{-t\Lambda^{\alpha}}u_0-\int_0^te^{-(t-\tau)\Lambda^{\alpha}}\nabla\cdot\left(u\nabla(-\Delta)^{-1}u\right)(\tau)d\tau
\end{equation}
in the metric space ($I=[0,T]$):
\begin{align*}
    \mathcal{D}_{T}:=\left\{ u:\ \|u\|_{\widetilde{L}^{\rho_{1}}_{T}(\dot{B}^{s_{1}}_{p,q})\cap\widetilde{L}^{\rho_{2}}_{T}(\dot{B}^{s_{2}}_{p,q})}\leq \eta, \ \ \ d(u,v) := \|u-v\|_{\widetilde{L}^{\rho_{1}}_{T}(\dot{B}^{s_{1}}_{p,q})\cap\widetilde{L}^{\rho_{2}}_{T}(\dot{B}^{s_{2}}_{p,q})}\right\}
\end{align*}
with
\begin{equation*}
     s_{1}=-1+\frac{n}{p}+\varepsilon,\ \ s_{2}=-1+\frac{n}{p}-\varepsilon, \ \ \rho_{1}=\frac{\alpha}{\alpha-1+\varepsilon},
     \ \ \rho_{2}=\frac{\alpha}{\alpha-1-\varepsilon}, \ \ 0<\varepsilon<\alpha-1.
\end{equation*}
Applying Proposition \ref{pro3.1} and Lemma \ref{le3.2} by choosing $\rho=\frac{\alpha}{2\alpha-2}$, for any $u,v\in\mathcal{D}_{T}$, we see that
\begin{align}\label{eq3.17}
    \|\mathbb{F}(u)\|_{\widetilde{L}^{\rho_{1}}_{T}(\dot{B}^{s_{1}}_{p,q})\cap\widetilde{L}^{\rho_{2}}_{T}(\dot{B}^{s_{2}}_{p,q})}
    &\lesssim \|e^{-t\Lambda^{\alpha}}u_{0}\|_{\widetilde{L}^{\rho_{1}}_{T}(\dot{B}^{s_{1}}_{p,q})\cap\widetilde{L}^{\rho_{2}}_{T}(\dot{B}^{s_{2}}_{p,q})}
    +\|u\nabla(-\Delta)^{-1}u\|_{\widetilde{L}^{\frac{\alpha}{2\alpha-2}}_{T}(\dot{B}^{-1+\frac{n}{p}}_{p,q})}\nonumber\\
    &\lesssim \|e^{-t\Lambda^{\alpha}}u_{0}\|_{\widetilde{L}^{\rho_{1}}_{T}(\dot{B}^{s_{1}}_{p,q})\cap\widetilde{L}^{\rho_{2}}_{T}(\dot{B}^{s_{2}}_{p,q})}
    +\|u\|_{\widetilde{L}^{\rho_{1}}_{T}(\dot{B}^{s_{1}}_{p,q})\cap\widetilde{L}^{\rho_{2}}_{T}(\dot{B}^{s_{2}}_{p,q})}^{2},
\end{align}
 and
\begin{equation}\label{eq3.18}
    d(\mathbb{F}(u), \mathbb{F}(v))
    \lesssim \eta d(u,v ) .
\end{equation}
 Based these two estimates \eqref{eq3.17} and \eqref{eq3.18}, applying the standard contraction mapping argument (cf. \cite{L02}), we can show that if we choose $T$ is properly small, then $\mathbb{F}$ is a contraction mapping from $(\mathcal{D}_{T}, d)$ into itself, we omit the details here. Therefore, there
 exists $u\in\mathcal{D}_{T}$  such that $\mathbb{F}(u)=u$, which is a unique solution of the system \eqref{eq1.1}. Moreover, by Proposition \ref{pro3.1}, we have
 \begin{equation*}
    \|u\|_{\widetilde{L}^{\infty}_{T}(\dot{B}^{-\alpha+\frac{n}{p}}_{p,q})}
    \lesssim \|u_{0}\|_{\dot{B}^{-\alpha+\frac{n}{p}}_{p,q}}
    +\|u\|_{\widetilde{L}^{\rho_{1}}_{T}(\dot{B}^{s_{1}}_{p,q})\cap\widetilde{L}^{\rho_{2}}_{T}(\dot{B}^{s_{2}}_{p,q})}^{2}
    \lesssim \|u_{0}\|_{\dot{B}^{-\alpha+\frac{n}{p}}_{p,q}}+\eta^{2}.
\end{equation*}
Thus the solution $u$ can be extended step by step and finally there is a maximal time $T^{*}$ such that
\begin{equation*}
  u\in \widetilde{L}^{\infty}(0,T^{*}; \dot{B}^{-\alpha+\frac{n}{p}}_{p,q}(\mathbb{R}^{n}))\cap\widetilde{L}^{\rho_{1}}(0,T^{*}; \dot{B}^{s_1}_{p,q}(\mathbb{R}^{n}))
     \cap\widetilde{L}^{\rho_{2}}(0,T^{*}; \dot{B}^{s_2}_{p,q}(\mathbb{R}^{n})).
\end{equation*}
If $T^{*}<\infty$ and $\|u\|_{\widetilde{L}^{\rho_{1}}_{T^{*}}(\dot{B}^{s_{1}}_{p,q})\cap\widetilde{L}^{\rho_{2}}_{T^{*}}(\dot{B}^{s_{2}}_{p,q})}<\infty$, we claim that the solution can be extended beyond the maximal time $T^{*}$. Indeed, let us consider the integral equation
\begin{equation}\label{eq3.19}
  u(t)=e^{-(t-T)\Lambda^{\alpha}}u(T)-\int_{T}^{t}e^{-(t-\tau)\Lambda^{\alpha}}\nabla\cdot(u\nabla(-\Delta)^{-1}u)(\tau)d\tau.
\end{equation}
As we have proved before, we can show that if we choose $T $ sufficiently close to $T^{*}$, then
\begin{align}\label{eq3.20}
    \|u(t)\|_{\widetilde{L}^{\rho_{1}}(T,T^{*}; \dot{B}^{s_{1}}_{p,q})\cap\widetilde{L}^{\rho_{2}}(T,T^{*}; \dot{B}^{s_{2}}_{p,q})}
    &\leq \|u(T)\|_{\widetilde{L}^{\rho_{1}}(T,T^{*}; \dot{B}^{s_{1}}_{p,q})\cap\widetilde{L}^{\rho_{2}}(T,T^{*}; \dot{B}^{s_{2}}_{p,q})}\nonumber\\
    &+\|u\|_{\widetilde{L}^{\rho_{1}}(T,T^{*}; \dot{B}^{s_{1}}_{p,q})\cap\widetilde{L}^{\rho_{2}}(T,T^{*}; \dot{B}^{s_{2}}_{p,q})}^{2}.
\end{align}
Note that \eqref{eq3.20} is analogous to \eqref{eq3.17}, which yields immediately that the solution exists on $[T,T^{*}]$. This is a contradiction to the fact that $T^{*}$ is  maximal.    Moreover, observe that
if $\|u_{0}\|_{\dot{B}^{-\alpha+\frac{n}{p}}_{p,q}}$ is sufficiently small, we can directly choose $T=\infty$ in \eqref{eq3.17} and \eqref{eq3.18},  which yields global well-posedness of \eqref{eq1.1} with small initial data.  We conclude the proof of  the first part of Theorem \ref{th1.1}.
\subsection{The case $1<p<\infty$: Gevrey analyticity}

In this subsection, we prove analyticity of the system \eqref{eq1.1} with initial data in $\dot{B}^{-\alpha+\frac{n}{p}}_{p,q}(\mathbb{R}^{n})$ with $1<\alpha\leq 2$ and $1<p<\infty$. We first recall the  following  two elementary results.

\begin{lemma}\label{le3.3}{\em (Lemma 3.2 in \cite{BBT13})}
Consider the operator $E_{\alpha}:=e^{-[(t-s)^{\frac{1}{\alpha}}+s^{\frac{1}{\alpha}}-t^{\frac{1}{\alpha}}]\Lambda_{1}}$ for $0\leq s\leq t$. Then $E_{\alpha}$ is either
the identity operator or is the Fourier multiplier with $L^{1}$ kernel whose $L^{1}$-norm is bounded independent of $s$ and $t$.
\end{lemma}

\begin{lemma}\label{le3.4}{\em (Lemma 3.3 in \cite{BBT13})}
Assume that the operator $F_{\alpha}:=e^{t^{\frac{1}{\alpha}}\Lambda_{1}-\frac{1}{2}t\Lambda^{\alpha}}$ for $t\geq0$. Then $F_{\alpha}$ is the Fourier multiplier which maps boundedly $L^{p}\rightarrow L^{p}$ for $1<p<\infty$,
and its operator norm is uniformly bounded with respect to $t\geq 0$.
\end{lemma}

\begin{proposition}\label{pro3.5}
Let $s\in \mathbb{R}$, $1<p<\infty$, $1\leq q,\rho_1\leq\infty$ and
$0<T\leq\infty$. Assume that $u_{0}\in
\dot{B}^{s}_{p,q}(\mathbb{R}^{n})$ and $f\in\widetilde{L}^{\rho_1}_{T}(e^{t^{\frac{1}{\alpha}}\Lambda_{1}}
\dot{B}^{s+\frac{\alpha}{\rho_{1}}-\alpha}_{p,q}(\mathbb{R}^{n}))$. Then \eqref{eq3.1} has
a unique solution $u\in\underset{\rho_1\leq
\rho\leq\infty}{\cap}\widetilde{L}^{\rho}_{T}(e^{t^{\frac{1}{\alpha}}\Lambda_{1}}\dot{B}^{s+\frac{\alpha}{\rho}}_{p,q}(\mathbb{R}^{n}))$. In addition, there exists a
constant $C>0$ depending only on $\alpha$ and $n$ such that for any $\rho_1\leq
\rho\leq\infty$, we have
\begin{equation}\label{eq3.21}
  \|u\|_{\widetilde{L}^{\rho}_{T}(e^{t^{\frac{1}{\alpha}}\Lambda_{1}}\dot{B}^{s+\frac{\alpha}{\rho}}_{p,q})}\leq
  C\big(\|u_{0}\|_{\dot{B}^{s}_{p,q}}+\|f\|_{\widetilde{L}^{\rho_1}_{T}(e^{t^{\frac{1}{\alpha}}\Lambda_{1}}\dot{B}^{s+\frac{\alpha}{\rho_{1}}-\alpha}_{p,q})}\big).
\end{equation}
\end{proposition}

\begin{proof}
Since Proposition \ref{pro3.1} has already ensured that \eqref{eq3.1} has a unique solution $u$, it suffices to prove that the inequality \eqref{eq3.23} holds. For this purpose, setting $U(t)=e^{t^{\frac{1}{\alpha}}\Lambda_{1}}u(t)$, then applying $\Delta_{j}e^{t^{\frac{1}{\alpha}}\Lambda_{1}}$ to \eqref{eq3.1} and taking $L^{p}$ norm to the resulting equality imply that
\begin{equation}\label{eq3.22}
   \|\Delta_{j}U(t)\|_{L^{p}}\leq \|e^{t^{\frac{1}{\alpha}}\Lambda_{1}-t\Lambda^{\alpha}}\Delta_{j}u_{0}\|_{L^{p}}
   +\|\int_{0}^{t}e^{t^{\frac{1}{\alpha}}\Lambda_{1}-(t-\tau)\Lambda^{\alpha}}\Delta_{j}f(\tau)d\tau\|_{L^{p}}.
\end{equation}
It follows from Lemmas \ref{le3.4} and \ref{le2.5} that there exists $\kappa>0$ such that
\begin{align}\label{eq3.23}
   \|e^{t^{\frac{1}{\alpha}}\Lambda_{1}-t\Lambda^{\alpha}}\Delta_{j}u_{0}\|_{L^{p}}
   &= \|e^{t^{\frac{1}{\alpha}}\Lambda_{1}-\frac{t}{2}\Lambda^{\alpha}}e^{-\frac{t}{2}\Lambda^{\alpha}}\Delta_{j}u_{0}\|_{L^{p}}\nonumber\\
   &\lesssim \|e^{-\frac{t}{2}\Lambda^{\alpha}}\Delta_{j}u_{0}\|_{L^{p}}\lesssim e^{-\kappa2^{\alpha j}t}\|\Delta_{j}u_{0}\|_{L^{p}}.
\end{align}
Notice the fact that we can rewrite
$$
  e^{t^{\frac{1}{\alpha}}\Lambda_{1}-(t-\tau)\Lambda^{\alpha}}=e^{-[(t-\tau)^{\frac{1}{\alpha}}
  +\tau^{\frac{1}{\alpha}}-t^{\frac{1}{\alpha}}]\Lambda_{1}
  +[(t-\tau)^{\frac{1}{\alpha}}\Lambda_{1}-\frac{t-\tau}{2}\Lambda^{\alpha}]-\frac{t-\tau}{2}\Lambda^{\alpha}}e^{\tau^{\frac{1}{\alpha}}\Lambda_{1}}.
$$
It follows from Lemmas \ref{le3.3}  and \ref{le3.4} that
\begin{align}\label{eq3.24}
  \|\int_{0}^{t}e^{t^{\frac{1}{\alpha}}\Lambda_{1}-(t-\tau)\Lambda^{\alpha}}\Delta_{j}f(\tau)d\tau\|_{L^{p}}
  &\lesssim\int_{0}^{t}\|e^{-\frac{t-\tau}{2}\Lambda^{\alpha}}\Delta_{j}e^{\tau^{\frac{1}{\alpha}}\Lambda_{1}}f(\tau)\|_{L^{p}}d\tau\nonumber\\
  &\lesssim\int_{0}^{t}e^{-\kappa(t-\tau)2^{\alpha j}}\|\Delta_{j}e^{\tau^{\frac{1}{\alpha}}\Lambda_{1}}f(\tau)\|_{L^{p}}d\tau.
\end{align}
Combining \eqref{eq3.23} and \eqref{eq3.24}, we see that
\begin{align}\label{eq3.25}
   \|\Delta_{j}U(t)\|_{L^{p}}
   \lesssim e^{-\kappa2^{\alpha j}t}\|\Delta_{j}u_{0}\|_{L^{p}}+\int_{0}^{t}e^{-\kappa(t-\tau)2^{\alpha j}}\|\Delta_{j}e^{\tau^{\frac{1}{\alpha}}\Lambda_{1}}f(\tau)\|_{L^{p}}d\tau.
\end{align}
Taking $L^{\rho}([0,T])$ norm to \eqref{eq3.25} and using Young's inequality,
\begin{align}\label{eq3.26}
   \|\Delta_{j}U(t)\|_{L^{\rho}_{T}(L^{p})}
   \lesssim \left(\frac{1-e^{-\kappa\rho2^{\alpha j} T}}{\kappa \rho2^{\alpha j}}\right)^{\frac{1}{\rho}}\|\Delta_{j}u_{0}\|_{L^{p}}
   +\left(\frac{1-e^{-\kappa \rho_{2}2^{\alpha j}T}}{\kappa \rho_{2}2^{\alpha j}}\right)^{\frac{1}{\rho_{2}}}\|\Delta_{j}e^{t^{\frac{1}{\alpha}}\Lambda_{1}}f(\tau)\|_{L^{\rho_{1}}_{T}(L^{p})},
\end{align}
where $\frac{1}{\rho}+1=\frac{1}{\rho_2}+\frac{1}{\rho_{1}}$. Finally, multiplying $2^{(s+\frac{\alpha}{\rho})j}$ and taking the $l^{q}$ norm to \eqref{eq3.26}, we conclude that
\begin{align*}
  \|U\|_{\widetilde{L}^{\rho}_{T}(\dot{B}^{s+\frac{\alpha}{\rho}}_{p,q})}
  &\lesssim
  \left[\sum_{j\in\mathbb{Z}}\left(\frac{1-e^{-\kappa\rho2^{\alpha j}T}}{\kappa \rho}\right)^{\frac{q}{\rho}}(2^{sj}\|\Delta_{j}u_{0}\|_{L^{p}})^{q}\right]^{\frac{1}{q}}\nonumber\\
  &+\left[\sum_{j\in\mathbb{Z}}\left(\frac{1-e^{-\kappa\rho_{2} 2^{\alpha j}T}}{\kappa \rho_{2}}\right)^{\frac{q}{\rho_{2}}}\left(2^{(s+\frac{\alpha}{\rho_{1}}-\alpha)}
  \|\Delta_{j}e^{t^{\frac{1}{\alpha}}\Lambda_{1}}f\|_{L^{\rho_{1}}_{T}(L^{p})}\right)^{q}\right]^{\frac{1}{q}}\nonumber\\
  &\lesssim\|u_{0}\|_{\dot{B}^{s}_{p,q}}+\|f\|_{\widetilde{L}^{\rho_1}(0,T;
  e^{t^{\frac{1}{\alpha}}\Lambda_{1}}\dot{B}^{s+\frac{\alpha}{\rho_1}-\alpha}_{p,q})},
\end{align*}
which leads to \eqref{eq3.21} .
\end{proof}

\medskip

We also need to establish the corresponding result as Lemma \ref{le3.2} in terms of the operator $e^{t^{\frac{1}{\alpha}}\Lambda_{1}}$.

\begin{lemma}\label{le3.6} Let $s>0$, $1<p<\infty$, $1\leq q, \rho, \rho_{1}, \rho_{2}\leq \infty$ with $\frac{1}{\rho}=\frac{1}{\rho_{1}}+\frac{1}{\rho_{2}}$. Then for any $\varepsilon>0$,  $0<T\leq\infty$, we have
\begin{align}\label{eq3.27}
     \|u\nabla(-\Delta)^{-1}v+v\nabla(-\Delta)^{-1}u\|_{\widetilde{L}^{\rho}_{T}(e^{t^{\frac{1}{\alpha}}\Lambda_{1}}\dot{B}^{s}_{p,q})}&\lesssim
     \|u\|_{\widetilde{L}^{\rho_{1}}_{T}(e^{t^{\frac{1}{\alpha}}\Lambda_{1}}\dot{B}^{s+\varepsilon}_{p,q})}
     \|v\|_{\widetilde{L}^{\rho_{2}}_{T}(e^{t^{\frac{1}{\alpha}}\Lambda_{1}}\dot{B}^{-1+\frac{n}{p}-\varepsilon}_{p,q})}\nonumber\\
     &+\|u\|_{\widetilde{L}^{\rho_{2}}_{T}(e^{t^{\frac{1}{\alpha}}\Lambda_{1}}\dot{B}^{-1+\frac{n}{p}-\varepsilon}_{p,q})}
     \|v\|_{\widetilde{L}^{\rho_{1}}_{T}(e^{t^{\frac{1}{\alpha}}\Lambda_{1}}\dot{B}^{s+\varepsilon}_{p,q})}.
\end{align}
Moreover, if we choose $\varepsilon=0$, then \eqref{eq3.27} also holds for $q=1$.
\end{lemma}

\begin{proof} Set $U(t)=e^{t^{\frac{1}{\alpha}}\Lambda_{1}}u(t)$, $V(t)=e^{t^{\frac{1}{\alpha}}\Lambda_{1}}v(t)$. Then, as  Lemma \ref{le3.2}, we use  Bony's paraproduct decomposition to get
\begin{align}\label{eq3.28}
    e^{t^{\frac{1}{\alpha}}\Lambda_{1}}\left(u\nabla(-\Delta)^{-1}v\!+\!v\nabla(-\Delta)^{-1}u\right)&=e^{t^{\frac{1}{\alpha}}\Lambda_{1}}
   \big(e^{-t^{\frac{1}{\alpha}}\Lambda_{1}}Ue^{-t^{\frac{1}{\alpha}}\Lambda_{1}}\nabla(-\Delta)^{-1}V
    \!+\!e^{-t^{\frac{1}{\alpha}}\Lambda_{1}}Ve^{-t^{\frac{1}{\alpha}}\Lambda_{1}}\nabla(-\Delta)^{-1}U\big)\nonumber\\
    & :=J_{1}+J_{2}+J_{3},
\end{align}
where
\begin{align*}
    J_{1}:&=e^{t^{\frac{1}{\alpha}}\Lambda_{1}}\sum_{j'\in\mathbb{Z}}e^{-t^{\frac{1}{\alpha}}\Lambda_{1}}\Delta_{j'}U
    e^{-t^{\frac{1}{\alpha}}\Lambda_{1}}\nabla(-\Delta)^{-1}S_{j'-1}V+e^{-t^{\frac{1}{\alpha}}\Lambda_{1}}\Delta_{j'}V
    e^{-t^{\frac{1}{\alpha}}\Lambda_{1}}\nabla(-\Delta)^{-1}S_{j'-1}U;\\
    J_{2}:&=e^{t^{\frac{1}{\alpha}}\Lambda_{1}}\sum_{j'\in\mathbb{Z}}e^{-t^{\frac{1}{\alpha}}\Lambda_{1}}S_{j'-1}U
    e^{-t^{\frac{1}{\alpha}}\Lambda_{1}}\nabla(-\Delta)^{-1}\Delta_{j'}V+e^{-t^{\frac{1}{\alpha}}\Lambda_{1}}S_{j'-1}V
    e^{-t^{\frac{1}{\alpha}}\Lambda_{1}}\nabla(-\Delta)^{-1}\Delta_{j'}U;\\
    J_{3}:&=e^{t^{\frac{1}{\alpha}}\Lambda_{1}}\sum_{j'\in\mathbb{Z}}\sum_{|j'-j''|\leq1}e^{-t^{\frac{1}{\alpha}}\Lambda_{1}}
    \Delta_{j'}Ue^{-t^{\frac{1}{\alpha}}\Lambda_{1}}\nabla(-\Delta)^{-1}\Delta_{j''}V+
    e^{-t^{\frac{1}{\alpha}}\Lambda_{1}}\Delta_{j'}Ve^{-t^{\frac{1}{\alpha}}\Lambda_{1}}\nabla(-\Delta)^{-1}\Delta_{j''}U.
\end{align*}
To estimate the terms  $J_{i}$ ($i=1,2,3$), we use an idea as in \cite{L00} and \cite{BBT12}, and consider the following bilinear operator $\mathcal{B}_{t}(f,g)$ of the form
\begin{align}\label{eq3.29}
  \mathcal{B}_{t}(f,g):&=e^{t^{\frac{1}{\alpha}}\Lambda_{1}}(e^{-t^{\frac{1}{\alpha}}\Lambda_{1}}fe^{-t^{\frac{1}{\alpha}}\Lambda_{1}}g)\nonumber\\
  &=\frac{1}{(2\pi)^{n}}\int_{\mathbb{R}^{n}}\int_{\mathbb{R}^{n}}
  e^{ix\cdot(\xi+\eta)}e^{t^{\frac{1}{\alpha}}(|\xi+\eta|_{1}-|\xi|_{1}-|\eta|_{1})}\hat{f}(\xi)\hat{g}(\eta)d\xi d\eta.
\end{align}
Note that we can split the domain of integration into sub-domains, depending on the sign of $\xi_{j}$, of $\eta_{j}$ and of $\xi_{j}+\eta_{j}$. Indeed, for $\varsigma=(\varsigma_{1},\cdots,\varsigma_{n})$, $\mu=(\mu_{1}, \cdots, \mu_{n})$, $\nu=(\nu_{1}, \cdots, \nu_{n})\in\mathbb{R}^{n}$ such that  $\varsigma_{i}$,
$\mu_{i}$, $\nu_{i}\in\{-1,1\}$, we denote
\begin{align*}
  &D_{\varsigma}:=\{\eta:\ \  \varsigma_{i}\eta_{i}\geq 0, \ \  i=1,2, \cdots, n\};\\
  &D_{\mu}:=\{\xi:\ \ \ \mu_{i}\xi_{i}\geq 0, \ \   i=1,2, \cdots, n\};\\
  &D_{\nu}:=\{\xi+\eta: \ \  \nu_{i}(\xi_{i}+\eta_{i})\geq 0, \ \  i=1,2, \cdots, n\}.
\end{align*}
Let $\chi_{D}$ be the characteristic function on the domain $D$. Then we can rewrite $\mathcal{B}_{t}(f,g)$  as
\begin{align*}
  \mathcal{B}_{t}(f,g)=\frac{1}{(2\pi)^{n}}\int_{\mathbb{R}^{n}}\int_{\mathbb{R}^{n}}
  e^{ix\cdot(\xi+\eta)}\chi_{D_{\nu}}e^{t^{\frac{1}{\alpha}}(|\xi+\eta|_{1}-|\xi|_{1}-|\eta|_{1})}\chi_{D_{\mu}}\hat{f}(\xi)\chi_{D_{\varsigma}}\hat{g}(\eta)d\xi d\eta.
\end{align*}
By this observation, we introduce the monodimensional  operators:
\begin{align*}
  K_{1}f:=\frac{1}{2\pi}\int_{0}^{+\infty}e^{ix\xi}\hat{f}(\xi)d\xi, \ \ \
  K_{-1}f:=\frac{1}{2\pi}\int_{-\infty}^{0}e^{ix\xi}\hat{f}(\xi)d\xi,
\end{align*}
and
\begin{align*}
  L_{t,\varepsilon_{1},\varepsilon_{2}}f:=f \ \ \ \text{if} \ \ \varepsilon_{1}\varepsilon_{2}=1,\ \ \
  L_{t,\varepsilon_{1},\varepsilon_{2}}f:=\frac{1}{2\pi}\int_{-\infty}^{+\infty}e^{ix\xi}e^{-2t^{\frac{1}{\alpha}}|\xi|_{1}}\hat{f}(\xi)d\xi
  \ \ \ \text{if} \ \ \varepsilon_{1}\varepsilon_{2}=-1.
\end{align*}
Moreover, for $t>0$, we define
the operators
\begin{align}\label{eq3.30}
  Z_{t,\varsigma,\mu}:=K_{\mu_{1}}L_{t,\varsigma_{1},\mu_{1}}\otimes\ldots\otimes K_{\mu_{n}}L_{t,\varsigma_{n},\mu_{n}}.
\end{align}
We mention here that the above tensor product \eqref{eq3.30} means that the $j-$th operator in the tensor product acts on
the $j-$th variable of the function $f(x_{1}, \ldots, x_{n})$. Then an elementary calculation yields the following
identity:
\begin{align}\label{eq3.31}
  \mathcal{B}_{t}(f,g)=\sum_{\varsigma,\mu,\nu\in\{-1,1\}^{n\times 3}}K_{\varsigma_{1}}\otimes\ldots\otimes K_{\varsigma_{n}}(Z_{t,\varsigma,\mu}fZ_{t,\varsigma,\nu}g).
\end{align}
Noticing that for $\xi+\eta\in D_{\nu}$, $\xi\in D_{\mu}$ and $\eta\in D_{\varsigma}$, $e^{t^{\frac{1}{\alpha}}(|\xi+\eta|_{1}-|\xi|_{1}-|\eta|_{1})}$ must belong to the following set:
\begin{equation*}
   \mathbb{E}:=\{1, e^{-2t^{\frac{1}{\alpha}}|\xi_{i}+\eta_{i}|_{1}}, e^{-2t^{\frac{1}{\alpha}}|\xi_{i}|_{1}}, e^{-2t^{\frac{1}{\alpha}}|\eta_{i}|_{1}}, \ \ i=1,2,\cdots, n \}.
\end{equation*}
Moreover, it is clear that $\chi_{D_{\varsigma}}$, $\chi_{D_{\mu}}$, $\chi_{D_{\nu}}\in\mathcal{M}_{p}$, and every element in $\mathbb{E}$ are the Fourier multipliers on $L^{p}(\mathbb{R}^{n})$ for $1<p<\infty$, which yield that the operators $K_{\varsigma}$ and $Z_{t,\varsigma, \mu}$ defined above are combinations of the identity operator
and of the Fourier multipliers on $L^{p}(\mathbb{R}^{n})$ (including Hilbert transform). Hence, the operators $K_{\varsigma}$ and $Z_{t,\varsigma,\mu}$  are bounded linear operators on $L^{p}(\mathbb{R}^{n})$ for $1<p<\infty$, and the corresponding operator norm of $Z_{t,\varsigma,\mu}$ is
bounded independent of $t\geq 0$. Moreover, for $1<p, p_{1}, p_{2}<\infty$,
\begin{equation*}
   \|\mathcal{B}_{t}(f,g)\|_{L^{p}}\lesssim \|Z_{t,\varsigma,\mu}fZ_{t,\varsigma,\nu}g\|_{L^{p}}\lesssim\|f\|_{L^{p_{1}}}\|g\|_{L^{p_{2}}}\ \ \ \text{with}\ \ \ \frac{1}{p_{1}}+\frac{1}{p_{2}}=\frac{1}{p}.
\end{equation*}
Since the nice boundedness property of  the bilinear operator $\mathcal{B}_{t}(f,g)$, we can follow  the proof of Lemma \ref{le3.2} to complete the proof of Lemma \ref{le3.4}. Indeed, we take the first term of $J_{1}$ as an example:
\begin{align*}
  \|\Delta_{j}e^{t^{\frac{1}{\alpha}}\Lambda_{1}}&\sum_{j'\in\mathbb{Z}}e^{-t^{\frac{1}{\alpha}}\Lambda_{1}}\Delta_{j'}U
    e^{-t^{\frac{1}{\alpha}}\Lambda_{1}}\nabla(-\Delta)^{-1}S_{j'-1}V\|_{L^{\rho}_{T}(L^{p})}\nonumber\\
   &= \|\Delta_{j}\sum_{j'\in\mathbb{Z}}\mathcal{B}_{t}(\Delta_{j'}U, \nabla(-\Delta)^{-1}S_{j'-1}V)\|_{L^{\rho}_{T}(L^{p})}\nonumber\\
   & \lesssim \sum_{|j'-j|\leq 4}\big\|K_{\varsigma_{1}}\otimes\ldots\otimes K_{\varsigma_{n}}(Z_{t,\varsigma,\mu}\Delta_{j'}UZ_{t,\varsigma,\nu}\nabla(-\Delta)^{-1}S_{j'-1}V)\big\|_{L^{\rho}_{T}(L^{p})}\nonumber\\
   & \lesssim \sum_{|j'-j|\leq 4}\big\|Z_{t,\varsigma,\mu}\Delta_{j'}U\|_{L^{\rho_{1}}_{T}(L^{p})}
   \|Z_{t,\varsigma,\nu}\nabla(-\Delta)^{-1}S_{j'-1}V\big\|_{L^{\rho_{2}}_{T}(L^{\infty})}\nonumber\\
    & \lesssim \sum_{|j'-j|\leq 4}\|Z_{t,\varsigma,\mu}\Delta_{j'}U\|_{L^{\rho_{1}}_{T}(L^{p})}\sum_{k\leq j'-2}2^{(-1+\frac{n}{p})k}\|Z_{t,\varsigma,\nu}\nabla(-\Delta)^{-1}\Delta_{k}V\|_{L^{\rho_{2}}_{T}(L^{p})}\nonumber\\
     &\lesssim \sum_{|j'-j|\leq 4}\|\Delta_{j'}U\|_{L^{\rho_{1}}_{T}(L^{p})}\sum_{k\leq j'-2}2^{\varepsilon k}2^{(-1+\frac{n}{p}-\varepsilon)k}\|\Delta_{k}V\|_{L^{\rho_{2}}_{T}(L^{p})}\nonumber\\
  &\lesssim \sum_{|j'-j|\leq 4}2^{-sj'}2^{(s+\varepsilon)j'}\|\Delta_{j'}U\|_{L^{\rho_{1}}_{T}(L^{p})}\|V\|_{\widetilde{L}^{\rho_{2}}_{T}(\dot{B}^{-1+\frac{n}{p}-\varepsilon}_{p,q})}.
\end{align*}
The other terms can be established analogously, thus we get the desired estimate \eqref{eq3.27}.
\end{proof}

\medskip

Combining Proposition \ref{pro3.5} and Lemma \ref{le3.6},  returning to the mapping $\eqref{eq3.16}$, we obtain
\begin{align}\label{eq3.32}
    \|\mathbb{F}(u)\|_{\widetilde{L}^{\rho_{1}}_{T}(e^{t^{\frac{1}{\alpha}}\Lambda_{1}}\dot{B}^{s_{1}}_{p,q})
    \cap\widetilde{L}^{\rho_{2}}_{T}(e^{t^{\frac{1}{\alpha}}\Lambda_{1}}\dot{B}^{s_{2}}_{p,q})}
    &\lesssim  \|e^{t^{\frac{1}{\alpha}}\Lambda_{1}-t\Lambda^{\alpha}}u_{0}\|_{\widetilde{L}^{\rho_{1}}_{T}(\dot{B}^{s_{1}}_{p,q})
    \cap\widetilde{L}^{\rho_{2}}_{T}(\dot{B}^{s_{2}}_{p,q})}\nonumber\\
    &\ \ \ +\|u\nabla(-\Delta)^{-1}u\|_{\widetilde{L}^{\rho_{1}}_{T}(e^{t^{\frac{1}{\alpha}}\Lambda_{1}}\dot{B}^{s_{1}}_{p,q})
    \cap\widetilde{L}^{\rho_{2}}_{T}(e^{t^{\frac{1}{\alpha}}\Lambda_{1}}\dot{B}^{s_{2}}_{p,q})}\nonumber\\
    &\lesssim  \|e^{-\frac{t}{2}\Lambda^{\alpha}}u_{0}\|_{\widetilde{L}^{\rho_{1}}_{T}(\dot{B}^{s_{1}}_{p,q})
    \cap\widetilde{L}^{\rho_{2}}_{T}(\dot{B}^{s_{2}}_{p,q})}\nonumber\\
    &\ \ \ +\|u\|_{\widetilde{L}^{\rho_{1}}_{T}(e^{t^{\frac{1}{\alpha}}\Lambda_{1}}\dot{B}^{s_{1}}_{p,q})
    \cap\widetilde{L}^{\rho_{2}}_{T}(e^{t^{\frac{1}{\alpha}}\Lambda_{1}}\dot{B}^{s_{2}}_{p,q})}^{2}.
\end{align}
Based on the above estimate \eqref{eq3.32},  by applying the standard contraction mapping argument, we complete the proof, as desired.

\subsection{The case $1<\alpha<2$ and $p=\infty$: Well-posedness}

In this subsection, we focus on the limit case  $p=\infty$. We first aim at establishing the following result.

\begin{lemma}\label{le3.7}  For $1\leq \alpha<2$, we have
\begin{align}\label{eq3.33}
     \|u\nabla(-\Delta)^{-1}v+v\nabla(-\Delta)^{-1}u\|_{\widetilde{L}^{1}_{t}(\dot{B}^{1-\alpha}_{\infty,1})}&\lesssim
     \|u\|_{\widetilde{L}^{\infty}_{t}(\dot{B}^{-\alpha}_{\infty,1})}
     \|v\|_{\widetilde{L}^{1}_{t}(\dot{B}^{0}_{\infty,1})}+\|u\|_{\widetilde{L}^{1}_{t}(\dot{B}^{0}_{\infty,1})}
     \|v\|_{\widetilde{L}^{\infty}_{t}(\dot{B}^{-\alpha}_{\infty,1})}.
\end{align}
\end{lemma}

\begin{proof}
Following from Lemma \ref{le3.2},  by applying H\"{o}lder's inequality, Lemmas \ref{le2.3} and \ref{le2.4}, we estimate the terms $I_{i}$ ($i=1,2,3$) as follows:
\begin{align*}
  \|\Delta_{j}I_{1}\|_{L^{1}_{t}(L^{\infty})}
 & \lesssim \sum_{|j'-j|\leq 4}\Big(\|\Delta_{j'}u\|_{L^{1}_{t}(L^{\infty})}\|\nabla(-\Delta)^{-1}S_{j'-1}v\|_{L^{\infty}_{t}(L^{\infty})}\nonumber\\
 &\ \ \ \ \ \ +
  \|\Delta_{j'}v\|_{L^{1}_{t}(L^{\infty})}\|\nabla(-\Delta)^{-1}S_{j'-1}u\|_{L^{\infty}_{t}(L^{\infty})}\Big)\nonumber\\
 & \lesssim \sum_{|j'-j|\leq 4}\Big(\|\Delta_{j'}u\|_{L^{1}_{t}(L^{\infty})}\sum_{k\leq j'-2}2^{(\alpha-1)k}2^{-\alpha k}\|\Delta_{k}v\|_{L^{\infty}_{t}(L^{\infty})}\nonumber\\&\ \ \ \ \ \ +
  \|\Delta_{j'}v\|_{L^{1}_{t}(L^{\infty})}\sum_{k\leq j'-2}2^{(\alpha-1)k}2^{-\alpha k}\|\Delta_{k}u\|_{L^{\infty}_{t}(L^{\infty})}\Big)\nonumber\\
  &\lesssim 2^{(\alpha-1)j}\sum_{|j'-j|\leq 4}\left(\|\Delta_{j'}u\|_{L^{1}_{t}(L^{\infty})}\|v\|_{\widetilde{L}^{\infty}_{t}(\dot{B}^{-\alpha}_{\infty,1})}
  +\|\Delta_{j'}v\|_{L^{1}_{t}(L^{\infty})}\|u\|_{\widetilde{L}^{\infty}_{t}(\dot{B}^{-\alpha}_{\infty,1})}\right).
\end{align*}
This along with Definition \ref{de2.2} leads to
\begin{align}\label{eq3.34}
      \|I_{1}\|_{\widetilde{L}^{1}_{t}(\dot{B}^{1-\alpha}_{\infty,1})}\lesssim
       \|u\|_{\widetilde{L}^{1}_{t}(\dot{B}^{0}_{\infty,1})}\|v\|_{\widetilde{L}^{\infty}_{t}(\dot{B}^{-\alpha}_{\infty,1})}
       +\|u\|_{\widetilde{L}^{\infty}_{t}(\dot{B}^{-\alpha}_{\infty,1})}\|v\|_{\widetilde{L}^{1}_{t}(\dot{B}^{0}_{\infty,1})}.
\end{align}
Similarly, for $I_{2}$,  we obtain
\begin{align*}
    \|\Delta_{j}I_{2}&\|_{L^{1}_{t}(L^{\infty})}
    \lesssim \sum_{|j'-j|\leq 4}2^{-j'}\sum_{k\leq j'-2}\Big(\|\Delta_{k}u\|_{L^{1}_{t}(L^{\infty})}\|\Delta_{j'}v\|_{L^{\infty}_{t}(L^{\infty})}
    +\|\Delta_{k}v\|_{L^{1}_{t}(L^{\infty})}\|\Delta_{j'}u\|_{L^{\infty}_{t}(L^{\infty})}\Big)\nonumber\\
     &\lesssim 2^{(\alpha-1)j}\sum_{|j'-j|\leq 4}2^{-\alpha j'}\Big(\|\Delta_{j'}v\|_{L^{\infty}_{t}(L^{\infty})}\|u\|_{\widetilde{L}^{1}_{t}(\dot{B}^{0}_{\infty,1})}
     +\|\Delta_{j'}u\|_{L^{\infty}_{t}(L^{\infty})}\|v\|_{\widetilde{L}^{1}_{t}(\dot{B}^{0}_{\infty,1})}\Big),
\end{align*}
which yields  directly to
\begin{align}\label{eq3.35}
  \|I_{2}\|_{\widetilde{L}^{1}_{t}(\dot{B}^{1-\alpha}_{\infty,1})}\lesssim
       \|u\|_{\widetilde{L}^{1}_{t}(\dot{B}^{0}_{\infty,1})}\|v\|_{\widetilde{L}^{\infty}_{t}(\dot{B}^{-\alpha}_{\infty,1})}
       +\|u\|_{\widetilde{L}^{\infty}_{t}(\dot{B}^{-\alpha}_{\infty,1})}\|v\|_{\widetilde{L}^{1}_{t}(\dot{B}^{0}_{\infty,1})}.
\end{align}
To treat with the remainder term $I_{3}$, as Lemma \ref{le3.2}, we split $I_{3}=K_{1}+K_{2}+K_{3}$  for $m=1,2,\cdots, n$,
 and consider $K_{1}$ and $K_{3}$
only. We infer from H\"{o}lder's inequality, Lemmas \ref{le2.3} and \ref{le2.4} that
\begin{align}\label{eq3.36}
    \|\Delta_{j}K_{1}\|_{L^{1}_{t}(L^{\infty})}&\lesssim 2^{2j}\sum_{j'\geq j-N_{0}}\sum_{|j'-j''|\leq1}
    \|(-\Delta)^{-1}\Delta_{j'}u\|_{L^{1}_{t}(L^{\infty})}\|\partial_{m}(-\Delta)^{-1}\Delta_{j''}v\|_{L^{\infty}_{t}(L^{\infty})}\nonumber\\
    &\lesssim 2^{2j}\sum_{j'\geq j-N_{0}}\sum_{|j'-j''|\leq1}2^{-2j'}
    \|\Delta_{j'}u\|_{L^{1}_{t}(L^{\infty})}2^{-j''}\|\Delta_{j''}v\|_{L^{\infty}_{t}(L^{\infty})}\nonumber\\
     &\lesssim 2^{2j}\sum_{j'\geq j-N_{0}}2^{(\alpha-3)j'}
    \|\Delta_{j'}u\|_{L^{1}_{t}(L^{\infty})}\|v\|_{\widetilde{L}^{\infty}_{t}(\dot{B}^{-\alpha}_{\infty,1})}\nonumber\\
    &\lesssim 2^{(\alpha-1)j}\sum_{j'\geq j-N_{0}}2^{(\alpha-3)(j'-j)}
    \|\Delta_{j'}u\|_{L^{1}_{t}(L^{\infty})}\|v\|_{\widetilde{L}^{\infty}_{t}(\dot{B}^{-\alpha}_{\infty,1})}.
\end{align}
\begin{align}\label{eq3.37}
    \|\Delta_{j}K_{3}\|_{L^{1}_{t}(L^{\infty})}&\lesssim 2^{j}\sum_{j'\geq j-N_{0}}\sum_{|j'-j''|\leq1}2^{-2j'}
    \|\Delta_{j'}u\|_{L^{1}_{t}(L^{\infty})}\|\Delta_{j''}v\|_{L^{\infty}_{t}(L^{\infty})}\nonumber\\
     &\lesssim 2^{j}\sum_{j'\geq j-N_{0}}2^{(\alpha-2)j'}
    \|\Delta_{j'}u\|_{L^{1}_{t}(L^{\infty})}\|v\|_{\widetilde{L}^{\infty}_{t}(\dot{B}^{-\alpha}_{\infty,1})}\nonumber\\
    &\lesssim 2^{(\alpha-1)j}\sum_{j'\geq j-N_{0}}2^{(\alpha-2)(j'-j)}
    \|\Delta_{j'}u\|_{L^{1}_{t}(L^{\infty})}\|v\|_{\widetilde{L}^{\infty}_{t}(\dot{B}^{-\alpha}_{\infty,1})}.
\end{align}
Under the assumption $1\leq\alpha<2$, we have
 $\alpha-3<0$ and $\alpha-2<0$. Hence,  putting the above estimates \eqref{eq3.36} and \eqref{eq3.37} together,  and multiplying $ 2^{(1-\alpha)j}$ to the resulting inequality, then taking $l^{1}$ norm implies that
\begin{align}\label{eq3.38}
      \| I_{3}\|_{\widetilde{L}^{1}_{t}(\dot{B}^{1-\alpha}_{\infty,1})}\lesssim
       \|u\|_{\widetilde{L}^{1}_{t}(\dot{B}^{0}_{\infty,1})}\|v\|_{\widetilde{L}^{\infty}_{t}(\dot{B}^{-\alpha}_{\infty,1})}.
\end{align}
Thanks to \eqref{eq3.34}, \eqref{eq3.35} and \eqref{eq3.38}, we get \eqref{eq3.33}. The proof of Lemma \ref{le3.7} is complete.
\end{proof}

\medskip

In order to prove the third part of Theorem \ref{th1.1}, we consider the resolution space $\widetilde{L}^{\infty}_{t}(\dot{B}^{-\alpha}_{\infty,1}(\mathbb{R}^{n}))\cap\widetilde{L}^{1}_{t}(\dot{B}^{0}_{\infty,1}(\mathbb{R}^{n}))$. Then,
 for the mapping \eqref{eq3.16},  we infer from Proposition \ref{eq3.1} and Lemma \ref{le3.7} that
\begin{align}\label{eq3.39}
      \| \mathbb{F}(u)\|_{\widetilde{L}^{\infty}_{t}(\dot{B}^{-\alpha}_{\infty,1})\cap\widetilde{L}^{1}_{t}(\dot{B}^{0}_{\infty,1})}&\lesssim \|u_{0}\|_{\dot{B}^{-\alpha}_{\infty,1}}+
       \|u\nabla(-\Delta)^{-1}u\|_{\widetilde{L}^{1}_{t}(\dot{B}^{1-\alpha}_{\infty,1})}\nonumber\\
       &\lesssim \|u_{0}\|_{\dot{B}^{-\alpha}_{\infty,1}}+
       \|u\|_{\widetilde{L}^{\infty}_{t}(\dot{B}^{-\alpha}_{\infty,1})\cap\widetilde{L}^{1}_{t}(\dot{B}^{0}_{\infty,1})}^{2}.
\end{align}
As before, applying the standard contraction mapping argument, we can show that if $\|u_{0}\|_{\dot{B}^{-\alpha}_{\infty,1}}$ is sufficiently small, then $\mathbb{F}$ is a contraction mapping from some suitable metric space into itself, this leads to that the system
\eqref{eq1.1} admits a unique solution in $u\in\widetilde{L}^{\infty}_{t}(\dot{B}^{-\alpha}_{\infty,1}(\mathbb{R}^{n}))\cap\widetilde{L}^{1}_{t}(\dot{B}^{0}_{\infty,1}(\mathbb{R}^{n}))$. We complete the proof, as desired.

\subsection{The case  $1<\alpha<2$ and $p=\infty$: Gevrey analyticity}
Set  $U(t):=e^{t^{\frac{1}{\alpha}}\Lambda_{1}}u(t)$.  Then $U(t)$ satisfies the following integral equation
\begin{equation}\label{eq3.40}
   U(t)=e^{t^{\frac{1}{\alpha}}\Lambda_{1}-t\Lambda^{\alpha}}u_{0}-\int_{0}^{t}\left[e^{t^{\frac{1}{\alpha}}\Lambda_{1}-(t-\tau)\Lambda^{\alpha}}\nabla\cdot
  \left (e^{-\tau^{\frac{1}{\alpha}}\Lambda_{1}}U\cdot e^{-\tau^{\frac{1}{\alpha}}\Lambda_{1}}\nabla(-\Delta)^{-1}U\right)\right](\tau)d\tau.
\end{equation}
Consider the linear part, since  the symbol $e^{t^{\frac{1}{\alpha}}|\xi|_{1}-\frac{t}{2}|\xi|^{\alpha}}$ is uniformly bounded for all $\xi$ and decays exponentially for $|\xi|\gg1$,  when localized in dyadic blocks in the Fourier spaces,  the Fourier multiplier $F_{\alpha}:=e^{t^{\frac{1}{\alpha}}\Lambda_{1}-\frac{1}{2}t\Lambda^{\alpha}}$ maps uniformly bounded from $L^{\infty}$ to $L^{\infty}$ for all $t\geq0$. Then,  by Young's inequality,  we have
\begin{align*}
  \|e^{t^{\frac{1}{\alpha}}\Lambda_{1}-t\Lambda^{\alpha}}u_{0}\|_{\widetilde{L}^{\infty}_{t}(\dot{B}^{-\alpha}_{\infty,1})
    \cap\widetilde{L}^{1}_{t}(\dot{B}^{0}_{\infty,1})}\lesssim
    \|e^{-\frac{1}{2}t\Lambda^{\alpha}}u_{0}\|_{\widetilde{L}^{\infty}_{t}(\dot{B}^{-\alpha}_{\infty,1})
    \cap\widetilde{L}^{1}_{t}(\dot{B}^{0}_{\infty,1})}\lesssim \|u_{0}\|_{\dot{B}^{-\alpha}_{\infty,1}}.
\end{align*}
For the nonlinear, by proceeding the same line as the proof of Lemma \ref{le3.6}, and observe that in general, the operators  $K_{\varsigma}$  and $Z_{t,\varsigma, \mu}$ defined in Lemma \ref{le3.6} do not map $L^{\infty}$ to  $L^{\infty}$ boundedly. However, when localized in dyadic blocks in the Fourier spaces, these operators are bounded in $L^{\infty}$.
Therefore, we can follow the calculations line by line from \eqref{eq3.34} to \eqref{eq3.38} in the proof of Lemma \ref{le3.7} to deal with the nonlinear term, and finally together with the estimate of the linear term ensure that
\begin{equation*}
    \|u(t)\|_{\widetilde{L}^{\infty}_{t}(e^{t^{\frac{1}{\alpha}}\Lambda_{1}}\dot{B}^{-\alpha}_{\infty,1})
    \cap\widetilde{L}^{1}_{t}(e^{t^{\frac{1}{\alpha}}\Lambda_{1}}\dot{B}^{0}_{\infty,1})}
       \lesssim \|u_{0}\|_{\dot{B}^{-\alpha}_{\infty,1}}+
       \|u(t)\|_{\widetilde{L}^{\infty}_{t}(e^{t^{\frac{1}{\alpha}}\Lambda_{1}}\dot{B}^{-\alpha}_{\infty,1})\cap\widetilde{L}^{1}_{t}(e^{t^{\frac{1}{\alpha}}\Lambda_{1}}\dot{B}^{0}_{\infty,1})}^{2}.
\end{equation*}
This completes the proof, as desired.
\subsection{Decay rate of solution}

In this subsection, we show the decay rate estimates of solutions obtained in Theorem \ref{th1.1}. The proof  is based on the following result.
\begin{lemma}\label{le3.8}
For all $\sigma\geq 0$ and $1<\alpha\leq 2$, the operator $\Lambda^{\sigma}e^{-t^{\frac{1}{\alpha}}\Lambda_{1}}$ is the convolution operator with a kernel $K_{\sigma}(t)\in L^{1}(\mathbb{R}^{n})$ for all $t>0$. Moreover,
\begin{equation}\label{eq3.41}
  \|K_{\sigma}(t)\|_{L^{1}}\leq C_{\sigma}t^{-\frac{\sigma}{\alpha}}.
\end{equation}
\end{lemma}
\begin{proof}
It suffices to consider the operator $\Lambda^{\sigma}e^{-\Lambda_{1}}$ and its kernel $\hat{k}_{\sigma}(\xi)=|\xi|^{\sigma}e^{-|\xi|_{1}}$ due to the general case can be obtained by using the scaling: $\xi\mapsto t^{\frac{1}{\alpha}}\xi$.  It is clear that $\hat{k}_{\sigma}(\xi)=|\xi|^{\sigma}e^{-|\xi|_{1}}\in L^{1}$. Thus $k_{\sigma}$ is a continuous bounded function.
Moreover, if $\sigma>0$, we introduce a function $\phi\in\mathcal{S}(\mathbb{R}^{n})$ so that $0\notin \operatorname{Supp}\phi$ and $\sum_{j\in\mathbb{Z}}\phi(2^{j}\xi)=1$. Then, $|\xi|^{\sigma}\phi(\xi)\in\mathcal{S}(\mathbb{R}^{n})$, and if we write
$|\xi|^{\sigma}\phi(\xi)=\hat{\Phi}_{\sigma}(\xi)$ and $\theta=1-\sum_{j\geq0}\phi(2^{j}\xi)$, then we have
\begin{equation*}
  \hat{k}_{\sigma}(\xi)=\sum_{j\geq0}2^{-j\sigma}\hat{\Phi}_{\sigma}(2^{j}\xi)e^{-|\xi|_{1}}+\theta(\xi)|\xi|^{\sigma}e^{-|\xi|_{1}}.
\end{equation*}
Hence,
\begin{equation*}
  \|k_{\sigma}\|_{L^{1}}\leq\sum_{j\geq0}2^{-j\sigma}\|\Phi_{\sigma}\|_{L^{1}}\|\mathcal{F}^{-1}(e^{-|\xi|_{1}})\|_{L^{1}}+
  \|\mathcal{F}^{-1}(\theta(\xi)|\xi|^{\sigma}e^{-|\xi|_{1}})\|_{L^{1}}<\infty.
\end{equation*}
We complete the proof of Lemma \ref{le3.8}.
\end{proof}

\medskip

Now the existence parts of Theorem \ref{th1.1} tell us that if the initial data $u_{0}$ is sufficiently small in critical
Besov space $\dot{B}^{-\alpha+\frac{n}{p}}_{p,q}(\mathbb{R}^{n})$ for either $1<\alpha\leq 2$, $1<p<\infty$ and $1\leq q\leq\infty$ or $1<\alpha<2$, $p=\infty$ and $q=1$, then the solution is  in the Gevrey class. Consequently, for all $\sigma\geq0$,
applying Lemma \ref{le3.8},  we get the following time decay of mild solution in terms of the homogeneous Besov-norm:
\begin{align}\label{eq3.42}
  \|\Lambda^{\sigma}u(t) \|_{\dot{B}^{-\alpha+\frac{n}{p}}_{p,q}}&
  =\|\Lambda^{\sigma}e^{-t^{\frac{1}{\alpha}}\Lambda_{1}}e^{t^{\frac{1}{\alpha}}\Lambda_{1}}u(t)\|_{\dot{B}^{-\alpha+\frac{n}{p}}_{p,q}}\nonumber\\
   &\leq C_{\sigma}t^{-\frac{\sigma}{\alpha}}
  \|e^{t^{\frac{1}{\alpha}}\Lambda_{1}}u(t)\|_{\dot{B}^{-\alpha+\frac{n}{p}}_{p,q}}\nonumber\\
  &\leq C_{\sigma}t^{-\frac{\sigma}{\alpha}}
  \|u_{0}\|_{\dot{B}^{-\alpha+\frac{n}{p}}_{p,q}}.
\end{align}
This completes the proof, as desired.

\section{The case $\alpha=1$: The proof of Theorem \ref{th1.2}}

In this section, we consider the case $\alpha=1$ of the system \eqref{eq1.1} with initial data in critical space $\dot{B}^{-1+\frac{n}{p}}_{p,1}(\mathbb{R}^{n})$ ($1\leq p\leq\infty$). The global well-posedness with small initial data and Gevrey analyticity will be established in the case that $1\leq p<\infty$ and $p=\infty$.

\subsection{The case $1\leq p<\infty$: Well-posedness}

We first recall some time-space estimates for solutions of the linear evolution
equation:
\begin{equation}\label{eq4.1}
\begin{cases}
  \partial_{t}u+\Lambda u= f(x,t), \ \
  &x\in\mathbb{R}^{n}, \ t>0,\\
  u(x,0)=u_{0}(x), \ \ &x\in\mathbb{R}^{n}.
\end{cases}
\end{equation}

\begin{proposition}\label{pro4.1}  {\em (\cite{HW13})}
Let $s\in \mathbb{R}$, $1\leq p,q\leq\infty$ and
$0<T\leq\infty$. Assume that $u_{0}\in
\dot{B}^{s}_{p,q}(\mathbb{R}^{n})$ and $f\in\widetilde{L}^{1}_{T}(\dot{B}^{s}_{p,q}(\mathbb{R}^{n}))$. Then \eqref{eq4.1} has
a unique solution $u\in\widetilde{L}^{\infty}_{T}(\dot{B}^{s}_{p,q}(\mathbb{R}^{n}))\cap\widetilde{L}^{1}_{T}(\dot{B}^{s+1}_{p,q}(\mathbb{R}^{n}))$. In addition, there exists a
constant $C>0$ depending only on $n$ such that
\begin{equation}\label{eq4.2}
  \|u\|_{\widetilde{L}^{\infty}_{T}(\dot{B}^{s}_{p,q})\cap\widetilde{L}^{1}_{T}(\dot{B}^{s+1}_{p,q})}\leq
  C\big(\|u_{0}\|_{\dot{B}^{s}_{p,q}}+\|f\|_{\widetilde{L}^{1}_{T}(\dot{B}^{s}_{p,q})}\big).
\end{equation}
\end{proposition}

Now for any initial data $u_{0}\in\dot{B}^{-1+\frac{n}{p}}_{p,1}(\mathbb{R}^{n})$, we consider the resolution space $\widetilde{L}^{\infty}_{t}(\dot{B}^{-1+\frac{n}{p}}_{p,1}(\mathbb{R}^{n}))$.
Slightly modifying the proof of Lemma \ref{le3.7}, we get the following result.

\begin{lemma}\label{le4.2} For any $u,v \in\widetilde{L}^{\infty}_{t}(\dot{B}^{-1+\frac{n}{p}}_{p,1}) $,  we have
\begin{equation}\label{eq4.3}
    \|u\nabla(-\Delta)^{-1}v+v\nabla(-\Delta)^{-1}u\|_{\widetilde{L}^{\infty}_{t}(\dot{B}^{-1+\frac{n}{p}}_{p,1})}\lesssim
       \|u\|_{\widetilde{L}^{\infty}_{t}(\dot{B}^{-1+\frac{n}{p}}_{p,1})}\|v\|_{\widetilde{L}^{\infty}_{t}(\dot{B}^{-1+\frac{n}{p}}_{p,1})}.
\end{equation}
\end{lemma}
\begin{proof}
We calculate the estimation of  $I_{1}$ as follows:
\begin{align}\label{eq4.4}
  \|\Delta_{j}I_{1}\|_{L^{\infty}_{t}(L^{p})}
  &\lesssim \sum_{|j'-j|\leq 4}\Big(\|\Delta_{j'}u\|_{L^{\infty}_{t}(L^{p})}\|\nabla(-\Delta)^{-1}S_{j'-1}v\|_{L^{\infty}_{t}(L^{\infty})}
 \nonumber\\
 &\ \ \ +\|\Delta_{j'}v\|_{L^{\infty}_{t}(L^{p})}\|\nabla(-\Delta)^{-1}S_{j'-1}u\|_{L^{\infty}_{t}(L^{\infty})}\Big)\nonumber\\
  &\lesssim \sum_{|j'-j|\leq 4}\Big(\|\Delta_{j'}u\|_{L^{\infty}_{t}(L^{p})}\sum_{k\leq j'-2}2^{(-1+\frac{n}{p})k}\|\Delta_{k}v\|_{L^{\infty}_{t}(L^{p})}\nonumber\\
  &\ \ \ +\|\Delta_{j'}v\|_{L^{\infty}_{t}(L^{p})}\sum_{k\leq j'-2}2^{(-1+\frac{n}{p})k}\|\Delta_{k}u\|_{L^{\infty}_{t}(L^{p})}\Big)\nonumber\\
  &\lesssim\sum_{|j'-j|\leq 4}\left(\|\Delta_{j'}u\|_{L^{\infty}_{t}(L^{p})}
  \|v\|_{\widetilde{L}^{\infty}_{t}(\dot{B}^{-1+\frac{n}{p}}_{p,1})}+\|\Delta_{j'}v\|_{L^{\infty}_{t}(L^{p})}
  \|u\|_{\widetilde{L}^{\infty}_{t}(\dot{B}^{-1+\frac{n}{p}}_{p,1})}\right).
\end{align}
Multiplying $2^{(-1+\frac{n}{p})j}$ to \eqref{eq4.4}, then taking $l^{1}$ norm to the resulting inequality, we get
\begin{align}\label{eq4.5}
  \|I_{1}\|_{\widetilde{L}^{\infty}_{t}(\dot{B}^{-1+\frac{n}{p}}_{p,1})}\lesssim
       \|u\|_{\widetilde{L}^{\infty}_{t}(\dot{B}^{-1+\frac{n}{p}}_{p,1})}\|v\|_{\widetilde{L}^{\infty}_{t}(\dot{B}^{-1+\frac{n}{p}}_{p,1})}.
\end{align}
Similarly, for $I_{2}$,
\begin{align}\label{eq4.6}
   \|\Delta_{j}I_{2}\|_{L^{\infty}_{t}(L^{p})}
   &\lesssim \sum_{|j'-j|\leq 4}\Big(\sum_{k\leq j'-2}2^{k}2^{(-1+\frac{n}{p})k}\|\Delta_{k}u\|_{L^{\infty}_{t}(L^{p})}2^{-j'}\|\Delta_{j'}v\|_{L^{\infty}_{t}(L^{p})}\nonumber\\
   &\ \ \ +\sum_{k\leq j'-2}2^{k}2^{(-1+\frac{n}{p})k}\|\Delta_{k}v\|_{L^{\infty}_{t}(L^{p})}2^{-j'}\|\Delta_{j'}u\|_{L^{\infty}_{t}(L^{p})}\Big)\nonumber\\
  &\lesssim \sum_{|j'-j|\leq 4}\left(\|\Delta_{j'}v\|_{L^{\infty}_{t}(L^{p})}\|u\|_{\widetilde{L}^{\infty}_{t}(\dot{B}^{-1+\frac{n}{p}}_{p,1})}+
  \|\Delta_{j'}u\|_{L^{\infty}_{t}(L^{p})}\|v\|_{\widetilde{L}^{\infty}_{t}(\dot{B}^{-1+\frac{n}{p}}_{p,1})}\right),
\end{align}
which leads directly to
\begin{align}\label{eq4.7}
  \|I_{2}\|_{\widetilde{L}^{\infty}_{t}(\dot{B}^{-1+\frac{n}{p}}_{p,1})}\lesssim
       \|u\|_{\widetilde{L}^{\infty}_{t}(\dot{B}^{-1+\frac{n}{p}}_{p,1})}\|v\|_{\widetilde{L}^{\infty}_{t}(\dot{B}^{-1+\frac{n}{p}}_{p,1})}.
\end{align}
Moreover, for the remainder term $I_{3}=K_{1}+K_{2}+K_{3}$  for $m=1,2,\cdots, n$.  In the case that $2\leq p<\infty$,
 $K_{1}$  and $K_{3}$ can be estimated  as follows ($K_{2}$ can be done analogously):
\begin{align}\label{eq4.8}
    \|\Delta_{j}K_{1}\|_{L^{\infty}_{t}(L^{p})}&\lesssim 2^{(2+\frac{n}{p})j}\sum_{j'\geq j-N_{0}}\sum_{|j'-j''|\leq1}
    \|(-\Delta)^{-1}\Delta_{j'}u\|_{L^{\infty}_{t}(L^{p})}\|\partial_{m}(-\Delta)^{-1}\Delta_{j''}v\|_{L^{\infty}_{t}(L^{p})}\nonumber\\
    &\lesssim 2^{(2+\frac{n}{p})j}\sum_{j'\geq j-N_{0}}\sum_{|j'-j''|\leq1}2^{-2j'}
    \|\Delta_{j'}u\|_{L^{\infty}_{t}(L^{p})}2^{-j''}\|\Delta_{j''}v\|_{L^{\infty}_{t}(L^{p})}\nonumber\\
    &\lesssim 2^{(2+\frac{n}{p})j}\sum_{j'\geq j-N_{0}}2^{-(1+\frac{2n}{p})j'}2^{(-1+\frac{n}{p})j'}
    \|\Delta_{j'}u\|_{L^{\infty}_{t}(L^{p})}2^{(-1+\frac{n}{p})j'}\|\Delta_{j'}v\|_{L^{\infty}_{t}(L^{p})}\nonumber\\
     &\lesssim2^{(1-\frac{n}{p})j}\sum_{j'\geq j-N_{0}}2^{-(1+\frac{2n}{p})(j'-j)}2^{(-1+\frac{n}{p})j'}
    \|\Delta_{j'}u\|_{L^{\infty}_{t}(L^{p})}
    \|v\|_{L^{\infty}_{t}(\dot{B}^{-1+\frac{n}{p}}_{p,1})}.
\end{align}
\begin{align}\label{eq4.9}
    \|\Delta_{j}K_{3}\|_{L^{\infty}_{t}(L^{p})}&\lesssim 2^{(1+\frac{n}{p})j}\sum_{j'\geq j-N_{0}}\sum_{|j'-j''|\leq1}
    \|(-\Delta)^{-1}\Delta_{j'}u\|_{L^{\infty}_{t}(L^{p})}\|\Delta_{j''}v\|_{L^{\infty}_{t}(L^{p})}\nonumber\\
    &\lesssim 2^{(1+\frac{n}{p})j}\sum_{j'\geq j-N_{0}}2^{-\frac{2n}{p}j'}2^{(-1+\frac{n}{p})j'}
    \|\Delta_{j'}u\|_{L^{\infty}_{t}(L^{p})}2^{(-1+\frac{n}{p})j'}\|\Delta_{j'}v\|_{L^{\infty}_{t}(L^{p})}\nonumber\\
    &\lesssim 2^{(1-\frac{n}{p})j}\sum_{j'\geq j-N_{0}}2^{(-\frac{2n}{p})(j'-j)}2^{(-1+\frac{n}{p})j'}
    \|\Delta_{j'}u\|_{L^{\infty}_{t}(L^{p})}\|v\|_{L^{\infty}_{t}(\dot{B}^{-1+\frac{n}{p}}_{p,1})}.
\end{align}
In the case that $1\leq p<2$, there exists $2\leq p'\leq\infty$ such that $\frac{1}{p}+\frac{1}{p'}=1$ such that
\begin{align}\label{eq4.10}
    \|\Delta_{j}K_{1}\|_{L^{\infty}_{t}(L^{p})}&\lesssim 2^{(2+n-\frac{n}{p})j}\sum_{j'\geq j-N_{0}}\sum_{|j'-j''|\leq1}
    \|(-\Delta)^{-1}\Delta_{j'}u\|_{L^{\infty}_{t}(L^{p'})}\|\partial_{m}(-\Delta)^{-1}\Delta_{j''}v\|_{L^{\infty}_{t}(L^{p})}\nonumber\\
    &\lesssim 2^{(2+n-\frac{n}{p})j}\sum_{j'\geq j-N_{0}}\sum_{|j'-j''|\leq1}2^{(-2+n(\frac{1}{p}-\frac{1}{p'}))j'}
    \|\Delta_{j'}u\|_{L^{\infty}_{t}(L^{p})}2^{-j''}\|\Delta_{j''}v\|_{L^{\rho_{1}}_{t}(L^{p})}\nonumber\\
    &\lesssim 2^{(2+n-\frac{n}{p})j}\sum_{j'\geq j-N_{0}}2^{-(n+1)j'}2^{(-1+\frac{n}{p})j'}
    \|\Delta_{j'}u\|_{L^{\infty}_{t}(L^{p})}2^{(-1+\frac{n}{p})j'}\|\Delta_{j'}v\|_{L^{\rho_{1}}_{t}(L^{p})}\nonumber\\
   & \lesssim 2^{(1-\frac{n}{p})j}\sum_{j'\geq j-N_{0}}2^{-(n+1)(j'-j)}2^{(-1+\frac{n}{p})j'}
    \|\Delta_{j'}u\|_{L^{\infty}_{t}(L^{p})}\|v\|_{L^{\infty}_{t}(\dot{B}^{-1+\frac{n}{p}}_{p,1})}.
\end{align}
\begin{align}\label{eq4.11}
    \|\Delta_{j}K_{3}\|_{L^{\infty}_{t}(L^{p})}&\lesssim 2^{(1+n-\frac{n}{p})j}\sum_{j'\geq j-N_{0}}\sum_{|j'-j''|\leq1}
    \|(-\Delta)^{-1}\Delta_{j'}u\|_{L^{\infty}_{t}(L^{p'})}\|\Delta_{j''}v\|_{L^{\infty}_{t}(L^{p})}\nonumber\\
    &\lesssim 2^{(1+n-\frac{n}{p})j}\sum_{j'\geq j-N_{0}}2^{-nj'}2^{(-1+\frac{n}{p})j'}
    \|\Delta_{j'}u\|_{L^{\infty}_{t}(L^{p})}2^{(-1+\frac{n}{p})j'}\|\Delta_{j'}v\|_{L^{\infty}_{t}(L^{p})}\nonumber\\
    &\lesssim 2^{(1-\frac{n}{p})j}\sum_{j'\geq j-N_{0}}2^{-n(j'-j)}2^{(-1+\frac{n}{p})j'}
    \|\Delta_{j'}u\|_{L^{\infty}_{t}(L^{p})}\|v\|_{L^{\infty}_{t}(\dot{B}^{-1+\frac{n}{p}}_{p,1})}.
\end{align}
Thus, putting the above estimates \eqref{eq4.8}--\eqref{eq4.11} together,  we obtain for all $1\leq p<\infty$,
\begin{align}\label{eq4.12}
      \| I_{3}\|_{\widetilde{L}^{\infty}_{t}(\dot{B}^{-1+\frac{n}{p}}_{p,1})}\lesssim
       \|u\|_{\widetilde{L}^{\infty}_{t}(\dot{B}^{-1+\frac{n}{p}}_{p,1})}\|v\|_{\widetilde{L}^{\infty}_{t}(\dot{B}^{-1+\frac{n}{p}}_{p,1})}.
\end{align}
Combining \eqref{eq4.5}, \eqref{eq4.7} and \eqref{eq4.12}, we conclude that \eqref{eq4.3} holds. The proof of Lemma \ref{le4.2} is complete.
\end{proof}

\medskip

Based on Proposition \ref{pro4.1} and Lemma \ref{le4.2}, consider the mapping \eqref{eq3.16},   we obtain
\begin{align}\label{eq4.13}
      \| \mathbb{F}(u)\|_{\widetilde{L}^{\infty}_{t}(\dot{B}^{-1+\frac{n}{p}}_{p,1})}&\lesssim \|u_{0}\|_{\dot{B}^{-1+\frac{n}{p}}_{p,1}}+
       \|u\nabla(-\Delta)^{-1}u\|_{\widetilde{L}^{\infty}_{t}(\dot{B}^{-1+\frac{n}{p}}_{p,1})}\nonumber\\
       &\lesssim \|u_{0}\|_{\dot{B}^{-1+\frac{n}{p}}_{p,1}}+
       \|u\|_{\widetilde{L}^{\infty}_{t}(\dot{B}^{-1+\frac{n}{p}}_{p,1})}^{2}.
\end{align}
Thus,  if $\|u_{0}\|_{\dot{B}^{-1+\frac{n}{p}}_{p,1}}$ is sufficiently small,  we can prove that $\mathbb{F}$ is a contraction mapping from some suitable metric space into itself, which implies that the system
\eqref{eq1.1} admits a unique solution in $\widetilde{L}^{\infty}_{t}(\dot{B}^{-1+\frac{n}{p}}_{p,1}(\mathbb{R}^{n}))$. The proof is complete, as desired.

\subsection{The case  $1< p<\infty$: Gevrey analyticity}
Note that when $\alpha=1$, the dissipation term $e^{-t\Lambda}$ is not strong enough to overcome the operator $e^{t\Lambda_{1}}$. Therefore, we need to define the Gevrey operator more carefully. Let
\begin{equation*}
   U(t): =e^{\frac{1}{2n}t\Lambda_{1}}u(t).
\end{equation*}
Then $U(t)$ satisfies the following integral equation
\begin{equation}\label{eq4.14}
   U(t)=e^{\frac{1}{2n}t\Lambda_{1}-t\Lambda}u_{0}-\int_{0}^{t}\left[e^{\frac{1}{2n}t\Lambda_{1}-(t-\tau)\Lambda}\nabla\cdot
  \left (e^{-\frac{1}{2n}\tau\Lambda_{1}}U\cdot e^{-\frac{1}{2n}\tau\Lambda_{1}}\nabla(-\Delta)^{-1}U\right)\right](\tau)d\tau.
\end{equation}
Notice that the operator $e^{\frac{1}{2n}t\Lambda_{1}-\frac{1}{2}t\Lambda}$
is a Fourier multiplier which maps uniformly bounded from $L^{p}(\mathbb{R}^{n})$ to $L^{p}(\mathbb{R}^{n})$ for $1<p<\infty$. Moreover, its operator norm is uniformly bounded with respect to any $t\geq0$ because the symbol $e^{\frac{1}{2n}t|\xi|_{1}-\frac{1}{2}t|\xi|}$ is uniformly bounded and decays exponentially for all $|\xi|\geq1$. Therefore, by Proposition \ref{pro4.1},  the linear term can be treated with
\begin{equation}\label{eq4.15}
   \|e^{\frac{1}{2n}t\Lambda_{1}-t\Lambda}u_{0}\|_{\widetilde{L}^{\infty}_{t}(\dot{B}^{-1+\frac{n}{p}}_{p,1})}\lesssim
   \|e^{-\frac{1}{2}t\Lambda}u_{0}\|_{\widetilde{L}^{\infty}_{t}(\dot{B}^{-1+\frac{n}{p}}_{p,1})}\lesssim\|u_{0}\|_{\dot{B}^{-1+\frac{n}{p}}_{p,1}}.
\end{equation}
For the nonlinear term, we rewrite
$$
  e^{\frac{1}{2n}t\Lambda_{1}-(t-\tau)\Lambda}=e^{\frac{1}{2n}(t-\tau)\Lambda_{1}-(t-\tau)\Lambda}e^{\frac{1}{2n}\tau\Lambda_{1}}.
$$
Thus, based on the nice boundedness properties of  the operator $e^{\frac{1}{2n}t\Lambda_{1}-\frac{1}{2}t\Lambda}$ and the bilinear operator $\mathcal{\widetilde{B}}_{t}(f,g)$ of the form
\begin{align*}
  \mathcal{\widetilde{B}}_{t}(f,g)&:=e^{\frac{1}{2n}t\Lambda_{1}}(e^{-\frac{1}{2n}t\Lambda_{1}}fe^{-\frac{1}{2n}t\Lambda_{1}}g),
\end{align*}
we can proceed along the lines of the proof of  Lemma \ref{le4.2} to obtain the Gevrey analyticity of the solution. Indeed, the bilinear operator $\mathcal{\widetilde{B}}_{t}(f,g)$ has a similar expression as \eqref{eq3.31}, moreover, the corresponding operators $\widetilde{K}_{\varsigma}$  and $\widetilde{Z}_{t,\varsigma,\mu}$ are bounded linear operators on $L^{p}(\mathbb{R}^{n})$ for $1<p<\infty$, and the corresponding operator norm of $\widetilde{Z}_{t,\varsigma,\mu}$ is bounded independent of $t\geq0$, thus, for $1<p, p_{1}, p_{2}<\infty$, we still have
\begin{equation*}
   \|\mathcal{\widetilde{B}}_{t}(f,g)\|_{L^{p}}\lesssim \|\widetilde{Z}_{t,\varsigma,\mu}f\widetilde{Z}_{t,\varsigma,\nu}g\|_{L^{p}}\lesssim\|f\|_{L^{p_{1}}}\|g\|_{L^{p_{2}}}\ \ \ \text{with}\ \ \ \frac{1}{p_{1}}+\frac{1}{p_{2}}=\frac{1}{p}.
\end{equation*}
This completes the proof, as desired.

\subsection{The case  $p=\infty$: Well-posedness}
Note that the resolution space $\widetilde{L}^{\infty}_{t}(\dot{B}^{-1}_{\infty,1}(\mathbb{R}^{n}))$ can not be adapted to the system \eqref{eq1.1} when $p=\infty$. Therefore, in the case  $p=\infty$, we consider the resolution space  $\widetilde{L}^{\infty}_{t}(\dot{B}^{-1}_{\infty,1}(\mathbb{R}^{n}))\cap\widetilde{L}^{1}_{t}(\dot{B}^{0}_{\infty,1}(\mathbb{R}^{n}))$.
Firstly, from Proposition \ref{pro4.1}, we  see that
\begin{equation}\label{eq4.16}
     \|e^{-t\Lambda}u_{0}\|_{\widetilde{L}^{\infty}_{t}(\dot{B}^{-1}_{\infty,1})\cap\widetilde{L}^{1}_{t}(\dot{B}^{0}_{\infty,1})}\lesssim
     \|u_{0}\|_{\dot{B}^{-1}_{\infty,1}}.
\end{equation}
Secondly, from Lemma \ref{le3.7}, we get
   \begin{align}\label{eq4.17}
     \|u\nabla(-\Delta)^{-1}u\|_{\widetilde{L}^{1}_{t}(\dot{B}^{0}_{\infty,1})}&\lesssim
     \|u\|_{\widetilde{L}^{\infty}_{t}(\dot{B}^{-1}_{\infty,1})\cap\widetilde{L}^{1}_{t}(\dot{B}^{0}_{\infty,1})}^{2}.
\end{align}
Hence, consider the mapping \eqref{eq3.16}, we deduce from  Proposition \ref{pro4.1},  \eqref{eq4.16} and \eqref{eq4.17} that
\begin{align}\label{eq4.18}
      \| \mathbb{F}(u)\|_{\widetilde{L}^{\infty}_{t}(\dot{B}^{-1}_{\infty,1})\cap\widetilde{L}^{1}_{t}(\dot{B}^{0}_{\infty,1})}&\lesssim \|u_{0}\|_{\dot{B}^{-1}_{\infty,1}}+
       \|u\nabla(-\Delta)^{-1}u\|_{\widetilde{L}^{1}_{t}(\dot{B}^{0}_{\infty,1})}\nonumber\\
       &\lesssim \|u_{0}\|_{\dot{B}^{-1}_{\infty,1}}+
       \|u\|_{\widetilde{L}^{\infty}_{t}(\dot{B}^{-1}_{\infty,1})\cap\widetilde{L}^{1}_{t}(\dot{B}^{0}_{\infty,1})}^{2}.
\end{align}
This reveals that, through the standard contraction mapping argument, if $\|u_{0}\|_{\dot{B}^{-1}_{\infty,1}}$ is sufficiently small, then $\mathbb{F}$ is a contraction mapping from some suitable metric space into itself, which means that the system
\eqref{eq1.1} admits a unique solution in $\widetilde{L}^{\infty}_{t}(\dot{B}^{-1}_{\infty,1}(\mathbb{R}^{n}))\cap\widetilde{L}^{1}_{t}(\dot{B}^{0}_{\infty,1}(\mathbb{R}^{n}))$.  The proof is complete, as desired.

\subsection{The case $p=\infty$: Gevrey analyticity}
To treat the Gevrey analyticity of solution in the case $p=\infty$, it suffices to prove that the following a priori estimate holds:
\begin{equation}\label{eq4.19}
    \|e^{\frac{1}{2n}t\Lambda_{1}}u(t)\|_{\widetilde{L}^{\infty}_{t}(\dot{B}^{-1}_{\infty,1})\cap\widetilde{L}^{1}_{t}(\dot{B}^{0}_{\infty,1})}
       \lesssim \|u_{0}\|_{\dot{B}^{-1}_{\infty,1}}+
       \|e^{\frac{1}{2n}t\Lambda_{1}}u(t)\|_{\widetilde{L}^{\infty}_{t}(\dot{B}^{-1}_{\infty,1})\cap\widetilde{L}^{1}_{t}(\dot{B}^{0}_{\infty,1})}^{2}.
\end{equation}
Since the symbol $e^{\frac{1}{2n}t|\xi|_{1}-\frac{1}{2}t|\xi|}$ is uniformly bounded in $\mathbb{R}^{n}$ and decays exponentially for sufficiently large $|\xi|\gg1$ with respect to all $t\geq0$, the estimation of linear part is straightforward due to  the fact that when localized in dyadic blocks in the Fourier spaces, the operator $e^{\frac{1}{2n}t\Lambda_{1}-\frac{1}{2}t\Lambda}$
 maps uniformly bounded from $L^{\infty}$ to $L^{\infty}$ with respect to $t\geq0$.  Thus,
\begin{equation}\label{eq4.20}
   \|e^{\frac{1}{2n}t\Lambda_{1}-t\Lambda}u_{0}\|_{\widetilde{L}^{\infty}_{t}(\dot{B}^{-1}_{\infty,1})\cap\widetilde{L}^{1}_{t}(\dot{B}^{0}_{\infty,1})}\lesssim
   \|e^{-\frac{1}{2}t\Lambda}u_{0}\|_{\widetilde{L}^{\infty}_{t}(\dot{B}^{-1}_{\infty,1})\cap\widetilde{L}^{1}_{t}(\dot{B}^{0}_{\infty,1})}\lesssim\|u_{0}\|_{\dot{B}^{-1}_{\infty,1}}.
\end{equation}
For the nonlinear part,  following the proofs of Lemmas \ref{le3.6} and \ref{le3.7}, the only difficulty arises from the following bilinear operator $\mathcal{\widetilde{B}}_{t}(f,g)$ of the form
\begin{align*}
  \mathcal{\widetilde{B}}_{t}(f,g)&=e^{\frac{1}{2n}t\Lambda_{1}}(e^{-\frac{1}{2n}t\Lambda_{1}}fe^{-\frac{1}{2n}t\Lambda_{1}}g)
\end{align*}
is not bounded from $L^{\infty}\times L^{\infty}$ to $L^{\infty}$, more precisely, the corresponding operators  $\widetilde{K}_{\varsigma}$  and $\widetilde{Z}_{t,\varsigma, \mu}$ in \eqref{eq3.31} do not map $L^{\infty}$ to  $L^{\infty}$ uniformly bounded. However, when localized in dyadic blocks in the Fourier spaces, these operators are bounded in $L^{\infty}$.
Therefore, we can follow the calculations line by line from \eqref{eq3.34} to \eqref{eq3.38}  with $\alpha=1$  in the proof of Lemma \ref{le3.7} to complete the estimation of the nonlinear term, which along with \eqref{eq4.20}, we arrive at \eqref{eq4.19}. The proof is complete, as desired.
\subsection{Decay rate of solution}
In this subsection, we show the decay rate estimates of solutions obtained in Theorem \ref{th1.2}. Based on Lemma \ref{le3.8}, we can show that
for all $\sigma\geq 0$, the operator $\Lambda^{\sigma}e^{-\frac{1}{2n}t\Lambda_{1}}$ is the convolution operator with a kernel $K_{\sigma}(t)\in L^{1}(\mathbb{R}^{n})$ for all $t>0$. Moreover,
\begin{equation}\label{eq4.21}
  \|K_{\sigma}(t)\|_{L^{1}}\leq \widetilde{C}_{\sigma}t^{-\sigma},
\end{equation}
where $\widetilde{C}_{\sigma}=\|\Lambda^{\sigma}e^{-\frac{1}{2n}\Lambda_{1}}\|_{L^{1}}$.
Now we know that the existence parts of Theorem \ref{th1.2} imply that if  $u_{0}\in \dot{B}^{-1+\frac{n}{p}}_{p,1}(\mathbb{R}^{n})$ ($1<p\leq\infty$) is sufficiently small, then the solution is  in the Gevrey class. Consequently, for all $\sigma\geq0$,
applying \eqref{eq4.21},  we get
\begin{align}\label{eq4.22}
  \|\Lambda^{\sigma}u(t) \|_{\dot{B}^{-1+\frac{n}{p}}_{p,1}}&
  =\|\Lambda^{\sigma}e^{-\frac{1}{2n}t\Lambda_{1}}e^{\frac{1}{2n}t\Lambda_{1}}u(t)\|_{\dot{B}^{-1+\frac{n}{p}}_{p,1}}\nonumber\\
   &\leq \widetilde{C}_{\sigma}t^{-\sigma}
  \|e^{\frac{1}{2n}t\Lambda_{1}}u(t)\|_{\dot{B}^{-1+\frac{n}{p}}_{p,1}}\nonumber\\
  &\leq \widetilde{C}_{\sigma}t^{-\sigma}
  \|u_{0}\|_{\dot{B}^{-1+\frac{n}{p}}_{p,1}}.
\end{align}
We complete the proof of Theorem \ref{th1.2}, as desired.

\medskip

\noindent{\bf Acknowledgements}.
The author would like to thank Prof. Song Jiang for his hospitality and Prof. Qiao Liu for many helpful discussions and suggestions.


\begin{thebibliography}{99}

%\bibitem{AA15} A. Atangana,  B. T. Alkahtani, Analysis of the Keller-Segel model with a fractional derivative
%without singular kernel, Entropy 17 (2015) , 4439--4453.

\bibitem{BBT12} H. Bae, A. Biswas, E. Tadmor, Analyticity and decay estimates of the Navier-Stokes equations in critical Besov spaces,
Arch. Rational Mech. Anal. 205 (2012)
963--991.

\bibitem{BBT13} H. Bae, A. Biswas, E. Tadmor, Analyticity of the subcritical and critical quasi-geostrophic equations in
Besov spaces, arXiv.1310.1624v1.

\bibitem{BCD11} H. Bahouri, J.-Y. Chemin, R. Danchin, \textit{Fourier Analysis and Nonlinear Partial Differential
Equations}, Grundlehren der mathematischen Wissenschaften, vol. 343.
Springer, Berlin, 2011.

\bibitem{B98} P. Biler, Local and global solvability of parabolic systems modelling
chemotaxis, Adv. Math. Sci. Appl. 8 (1998)  715--743.

\bibitem{BCGK04} P. Biler, M. Cannone, I.A. Guerra, G. Karch, Global regular and singular solutions for a model
of gravitating particles, Math. Ann. 330 (2004) 693--708.

\bibitem{BHN94} P. Biler, D. Hilhorst, T. Nadzieja, Existence and nonexistence of solutions for a model gravitational of particles, Colloq. Math.
67 (1994) 297--308.

\bibitem{BK10} P. Biler, G. Karch, Blowup of solutions to generalized Keller-Segel model, J. Evol. Equ. 10 (2010) 247--262.

\bibitem{BKLN961}  P. Biler, G. Karch, P. Lauren\c{c}ot, T. Nadzieja, The $8\pi$ problem for radially symmetric solutions of a chemotaxis model in a disc,
Topol. Methods Nonlinear Anal. 27 (2006) 133--144.

\bibitem{BKLN962}   P. Biler, G. Karch, P. Lauren\c{c}ot, T. Nadzieja, The $8\pi$ problem for radially symmetric solutions of a chemotaxis model in the
plane, Math. Methods Appl. Sci. 29 (2006) 1563--1583.

\bibitem{BKZ14} P. Biler, G. Karch, J. Zienkiewicz, Optimal criteria for blowup of radial solutions of chemotaxis systems, arXiv:1407.4501.

%\bibitem{BW98} P. Biler, W. A. Woyczy\'{n}ski, Global and exploding solutions
%for nonlocal quadratic evolution problems, SIAM J. Appl. Math. 59 (1998) 845--869.

\bibitem{BW09} P. Biler, G. Wu,
Two-dimensional chemotaxis models with fractional diffusion, Math.
Methods Appl. Sci. 32 (2009) 112--126.

%\bibitem{B12} A. Biswas, Gevrey regularity for a class of dissipative equations with applications to decay,
%J. Differential Equations 253(10) (2012) 2739--2764.

\bibitem{B14} A. Biswas, Gevrey regularity for the supercritical quasi-geostrophic equation,
J. Differential Equations 257 (2014) 1753--1772.

\bibitem{BS07} A. Biswas, D. Swanson, Existence and generalized Gevrey regularity of solutions to the Kuramoto-Sivashinsky
equation in $\mathbb{R}^{n}$, J. Differential Equations 240(1) (2007) 145--163.

\bibitem{BCM07} A. Blanchet, J.A. Carrillo, N. Masmoudi, Infinite time aggregation
for the critical Patlak-Keller-Segel model in $\mathbb{R}^{2}$, Comm. Pure Appl. Math.  61(10) (2008) 1449--1481.

\bibitem{BDP06} A. Blanchet, J. Dolbeault, B. Perthame, Two dimensional
Keller-Segel model: Optimal critical mass and qualitative properties of
the solutions, Electron. J. Diff. Equ. 44 (2006) 1--33.

\bibitem{B81} J.M. Bony, Calcul symbolique et propagation des singularit\'{e}s pour les \'{e}quations aux d\'{e}riv\'{e}es partielles non
lin\'{e}aires,  Ann. Sci. \'{e}cole Norm. Sup. 14(4) (1981) 209--246.

\bibitem{BP08} J. Bourgain, N. Pavlovi\'{c}, Ill-posedness of the Navier-Stokes equations in a critical space in 3D, J. Funct. Anal. 255(9) (2008) 2233--2247.

\bibitem{CP84} S. Childress, J.K. Percus, \textit{Chemotactic Collapse in Two Dimensions}, Lecture
Notes in Biomathematics, Vol. 55, Springer, Berlin-Heidelberg-New York, pp. 61--66, 1984.

\bibitem{CPZ04} L. Corrias, B. Perthame, H. Zaag, Global solutions of some chemotaxis and angiogenesis system in high space dimensions, Milan J. Math. 72 (2004)
1--28.


%\bibitem{C99} J.-Y. Chemin, Th\'{e}or\`{e}mes d'unicit\'{e} pour le
%syst\`{e}me de Navier-Stokes tridimensionnel, J. Anal. Math. \textbf{77} (1999) 27--50.

%\bibitem{CL95} J.-Y. Chemin, N. Lerner, Flot de champs de vecteurs non lipschitziens et \'{e}quations de Navier-Stokes,
%J. Differential Equations \textbf{121} (1995) 314--328.

%\bibitem{DL13} C. Deng, C. Li, Endpoint bilinear estimates and applications to the two-dimensional Poisson-Nernst-Planck system,
%Nonlinearity \textbf{26} (2013) 2993--3009.

\bibitem{DNR98} J.I. Diaz, T. Nagai, J.M. Rakotoson, Symmetrization techniques on unbounded domains: application to a chemotaxis system on $\mathbb{R}^{n}$, J. Differential Equations
145 (1998) 156--183.

\bibitem{E06} C. Escudero, The fractional Keller-Segel model,
Nonlinearity 19 (2006) 2909--2918.

%\bibitem{FT79} C. Foias, R. Temam, Some analytic and geometric properties of the solutions of the Navier-Stokes equations,
%J. Math. Pures Appl. 59 (1979) 339--368.

\bibitem{FT89} C. Foias, R. Temam, Gevrey class regularity for the solutions of the Navier-Stokes equations, J. Funct. Anal. 87 (1989) 359--369.

\bibitem{GZ98} H. Gajewski, K. Zacharias, Global behaviour of a reaction-diffusion system modelling
chemotaxis, Math. Nachr. 195 (1998) 77--114.

%\bibitem{GK98} Z. Gruji\v{c}, I. Kukavica, Space analyticity for the Navier-Stokes and related equations with initial data in $L^{p}$, J. Funct. Anal.
%152 (1998) 447--466.

\bibitem{HV961} M.A. Herrero, J.J.L. Val\'{a}zquez,  Singularity patterns in a chemotaxis model, Math. Ann.  306(3) (1996) 583--623.

\bibitem{HV962} M.A. Herrero, J.J.L. Val\'{a}zquez, Chemotaxis
collapse for the Keller-Segel model, J. Math. Biol. 35 (1996)
177--194.

\bibitem{H03} D. Horstmann, From 1970 until present: the Keller-Segel model in chemotaxis and its consequences I, Jahresber. DMV 105 (2003) 103--165.

\bibitem{H04} D. Horstmann, From 1970 until present: the Keller-Segel model in chemotaxis and its consequences II, Jahresber. DMV 106 (2004) 51--69.

\bibitem{HW13} C. Huang, B. Wang, Analyticity for the (generalized) Navier-Stokes equations with rough initail data, arXiv.13102.2141v2.

\bibitem{I11} T. Iwabuchi, Global well-posedness for Keller-Segel system in Besov type spaces, J. Math. Anal. Appl. 379 (2011) 930--948.

\bibitem{JL92} W. J\"{a}ger, S. Luckhaus, On explosions of solutions to a system of partial differential equations modelling chemotaxis, Trans. Amer. Math. Soc. 329(2) (1992) 819--824.

%\bibitem{K99} G. Karch,
%Scaling in nonlinear parabolic equations, J. Math. Anal. Appl. 234
%(1999) 534--558.

\bibitem{KS70} E.F. Keller, L.A. Segel, Initiation of slime mold aggregation viewed as an instability, J. Theoret. Biol. 26 (1970) 399--415.

\bibitem{KS08} H. Kozono, Y. Sugiyama, Local existence and finite time blow-up of solutions in the 2-D Keller-Segel system, J. Evol. Equ. 8 (2008) 353--378.

%\bibitem{KO08} M. Kurokiba, T. Ogawa, Well-posedness for the drift-diffusion system in $L^{p}$ arising from the semiconductor device
%simulation, J. Math. Anal. Appl. 342 (2008) 1052--1067.

\bibitem{L00} P.-G. Lemari\'{e}-Rieusset, On the analyticity of mild solutions for the Navier-Stokes equations,
C. R. Acad. Sci. Paris, Ser I 330 (2000) 183--186.

\bibitem{L02} P.-G. Lemari\'{e}-Rieusset,
\textit{Recent Developments in the Navier-Stokes Problem}, Research
Notes in Mathematics, Chapman \& Hall/CRC, 2002.

\bibitem{L13} P.-G. Lemari\'{e}-Rieusset, Small data in an optimal Banach space for the parabolic-parabolic
and parabolic-elliptic Keller-Segel equations in the whole space, Adv. Differ. Equ. 18 (2013)
1189--1208.

\bibitem{LR091} D. Li, J. Rodrigo, Finite-time singularities of an aggregation
equation in $\mathbb{R}^{n}$ with fractional dissipation,  Comm. Math. Phys. 287
(2009)  687--703.

\bibitem{LR092} D. Li, J. Rodrigo,  Refined blowup criteria and nonsymmetric
blowup of an aggregation equation, Adv. Math. 220 (2009) 1717--1738.

\bibitem{MYZ08} C. Miao, B. Yuan, B. Zhang, Well-posedness of the Cauchy problem for the
fractional power dissipative equations, Nonlinear Anal. 68 (2008)
461--484.

%\bibitem{N95} T. Nagai, Blow-up of radially symmetric solutions to a
%chemotaxis system, Adv. Math. Sci. Appl. 5 (1995), 581--601.

\bibitem{N00} T. Nagai, Behavior of solution to a parabolic-elliptic system modelling chemotaxis, J. Korean Math. Soc. 37 (2000) 721--733.

\bibitem{NSY97} T. Nagai, T. Senba, K. Yoshida, Application of the Trudinger-Moser inequality to a parabolic system of chemotaxis, Funkcial. Ekvac. 40(3) (1997) 411--433.

\bibitem{OS08} T. Ogawa, S. Shimizu, The drift-diffusion system in two-dimensional critical Hardy space,
J. Funct. Anal. 255 (2008) 1107--1138.

\bibitem{OS10} T. Ogawa, S. Shimizu, End-point maximal regularity and wellposedness of
the two dimensional Keller-Segel system in a critical Besov space,
Math. Z. 264 (2010) 601--628.

%\bibitem{OY09} T. Ogawa, M. Yamamoto , Asymptotic behavior of solutions to  drift-diffusion system with generalized dissipation,
%Math. Models Methods Appl. Sci. 19 (2009) 939--967.

%\bibitem{T83} H. Triebel, \textit{Theory of Function Spaces}, Birkh\"{a}user-Verlag, 1983.

\bibitem{W15} B. Wang, Ill-posedness for the Navier-Stokes equations incritical Besov spaces $\dot{B}^{-1}_{\infty,q}$,  Adv. Math. 268 (2015) 350--372.

\bibitem{WY08} G. Wu, J.Yuan, Well-posedness of the Cauchy problem for the fractional
power dissipative equation in critical Besov spaces, J. Math. Anal.
Appl. 340 (2008) 1326--1335.

\bibitem{WZ11} G. Wu, X. Zheng, On the well-posedness for Keller-Segel system with fractional diffusion,
Math. Methods Appl. Sci. 34(14) (2011) 1739--1750.

\bibitem{Y11} M. Yamamoto, Spatial analyticity of solutions to the drift-diffusion
equation with generalized dissipation, Arch. Math. 97 (2011) 261--270.

\bibitem{Y97} A. Yagi, Norm behavior of solutions to a parabolic system of chemotaxis, Math. Japan. 45 (1997) 241--265.

\bibitem{Y10} T. Yoneda, Ill-posedness of the 3D Navier-Stokes equations in a generalized Besov space near $BMO^{-1}$, J. Funct. Anal. 258(10) (2010) 3376--3387.

\bibitem{Z10} Z. Zhai, Global well-posedness for nonlocal fractional
Keller-Segel systems in critical Besov spaces, Nonlinear Anal. 72
(2010) 3173--3189.

\bibitem{ZLC11} J. Zhao, Q. Liu, S. Cui, Regularizing and decay rate estimates for solutions to
the Cauchy problem of the Debye-H\"{u}ckel system, Nonlinear Differ.
Equ. Appl. 19 (2012) 1--18.

%\bibitem{ZLC13} J. Zhao, Q. Liu, S. Cui, Existence of solutions for the Debye-H\"{u}ckel system with low regularity initial data,
%Acta Appl Math. \textbf{125} (2013) 1--10.
\end{thebibliography}
\end{document}